\documentclass[oneside,english,reqno]{amsart}
\usepackage{etex}
\usepackage{geometry}
\usepackage{adjustbox}
\usepackage{marvosym}
\usepackage{amsthm}
\usepackage{amssymb}
\usepackage{setspace}
\usepackage{tikz}
\usepackage{hyperref}
\usepackage{enumitem}
\usepackage{subfigure}

\usetikzlibrary{shapes.geometric}

\usepackage{mathrsfs}
\usetikzlibrary{arrows}

\usetikzlibrary{decorations.markings}
\tikzset{->-/.style={decoration={
  markings,
  mark=at position #1 with {\arrow{>}}},postaction={decorate}}}
  
\usepackage{lscape}

\makeatletter
\renewcommand{\p@subfigure}{}
\renewcommand{\@thesubfigure}{(\alph{subfigure})\hskip\subfiglabelskip}
\makeatother

\numberwithin{equation}{section}
\numberwithin{figure}{section}
\numberwithin{table}{section}
\theoremstyle{plain}
\newtheorem{theorem}{Theorem}
  \newtheorem{cor}[theorem]{Corollary}

  \newtheorem{question}[theorem]{Question}
  \theoremstyle{definition}
  
  \newtheorem{prop}[theorem]{Proposition}
  \numberwithin{theorem}{section}

\renewcommand{\P}{\mathcal{P}}
\newcommand{\Q}{\mathcal{Q}}
\newcommand{\G}{\mathcal{G}}
\newcommand{\K}{\mathcal{K}}
\newcommand{\B}{\mathcal{B}}
\newcommand{\C}{\mathcal{C}}
\newcommand{\F}{\mathcal{F}}

\newcommand{\up}{\mathcal{U}}
\newcommand{\down}{\mathcal{D}}

\newcommand{\cH}{\mathcal{H}}
\renewcommand{\diamond}{\diamondsuit}
\newcommand{\lu}{\operatorname{lu}}
\newcommand{\A}{\mathcal{A}}

\newcommand{\OR}{\operatorname{OR}}
\newcommand{\BR}{\operatorname{BR}}

\newcommand{\R}{\operatorname{R}}
\newcommand{\CR}{\operatorname{CR}}
\newcommand{\PR}[1]{\operatorname{R}_{#1}}

\newcommand{\levels}{e}

\begin{document}

\title{Ramsey numbers for partially-ordered sets}

\author{Christopher Cox$^1$ \and Derrick Stolee$^2$}

\date{\today}

\begin{abstract}
We present a refinement of Ramsey numbers by considering graphs with a partial ordering on their vertices.
This is a natural extension of the ordered Ramsey numbers.
We formalize situations in which we can use arbitrary families of partially-ordered sets to form host graphs for Ramsey problems.
We explore connections to well studied Tur\'an-type problems in partially-ordered sets, particularly those in the Boolean lattice.
We find a strong difference between Ramsey numbers on the Boolean lattice and ordered Ramsey numbers when the partial ordering on the graphs have large antichains.
\end{abstract}

\maketitle

\footnotetext[1]{Carnegie Mellon University, Pittsburgh, PA, USA. \texttt{cocox@andrew.cmu.edu}.}
\footnotetext[2]{Microsoft, Raleigh, NC, USA. \texttt{dstolee@microsoft.com}; Work completed while at Iowa State University.}

\section{Introduction}

Ramsey and Tur\'an problems are fundamental to graph theory.
Tur\'an problems focus on the maximum size of objects that forbid a certain substructure whereas Ramsey problems concern partitioning an object into parts where each part forbids a certain substructure.
Traditionally, these problems are considered in the domain of graphs.
Recently, Ramsey problems have been extended to graphs with a total ordering on their vertices~\cite{BCKK13,CP02,CGKVV14,CFLS14,CS15,FPSS12,MSW15,MS14}, and Tur\'an problems have been considered within the Boolean lattice~\cite{DK07,DKS05,GL13,GL15,GL09,GMT14,KMY13}.
We unite and generalize these concepts into Ramsey theory on partially-ordered sets. 

Ramsey numbers describe the transition where it becomes impossible to partition a complete graph into $t$ parts such that each part does not contain a certain subgraph.
For $k$-uniform hypergraphs $G_1,\dots,G_t$, the \emph{$t$-color graph Ramsey number} $\R^k(G_1,\dots,G_t)$ is the least integer $N$ such that any $t$-coloring of the edges of the $k$-uniform complete graph on $N$ vertices contains a copy of $G_i$ in color $i$ for some $i \in \{1,\dots,t\}$; when $G_1 = \cdots G_t = G$, we shorten the notation to $\R_t^k(G)$.
Since $\R_t^k(K_n)$ is finite for all $t$ and $n$, all Ramsey numbers exist, including the generalizations we discuss in this paper.
In our notation for Ramsey numbers, we use $k$ to emphasize that $G_1,\dots,G_t$ are $k$-uniform graphs.

A $k$-uniform \emph{ordered hypergraph} is a $k$-uniform hypergraph $G$ with a total order on the vertex set $V(G)$.
An ordered hypergraph $G$ \emph{contains} another ordered hypergraph $H$ exactly when there exists an embedding of $H$ in $G$ that preserves the vertex order.
For ordered $k$-uniform hypergraphs $G_1,\dots,G_t$, the \emph{ordered Ramsey number} $\OR^k(G_1,\dots,G_t)$ is the least integer $N$ such that every $t$-coloring of the edges of the complete $k$-uniform graph with vertex set $\{1,\dots,N\}$ contains an ordered copy of $G_i$ in color $i$ for some $i \in \{1,\dots,t\}$.
Since there is essentially one ordering of the complete graph, $\OR^k(G_1,\dots,G_t) \leq \R_t^k(K_n)$ for $n = \max\{ |V(G_i)| : i \in \{1,\dots,t\}\}$.
In general, $\OR_t^k(G)$ can be much larger than $\R_t^k(G)$, such as when $G$ is an ordered path.
Ordered Ramsey numbers on ordered paths have deep connections to the Erd\H{o}s-Szekeres Theorem and the Happy Ending Problem~\cite{ES35} (see~\cite{FPSS12,MS14}).

A \emph{partially-ordered set}, or \emph{poset}, is a pair $(X,\leq)$ where  $X$ is a set and $\leq$ is a relation such that $\leq$ is reflexive, anti-symmetric, and transitive.
A pair $x,y \in X$ is \emph{comparable} if $x\leq y$ or $y \leq x$, and a \emph{$k$-chain} is a set of $k$ distinct, pairwise comparable elements.
If $P$ and $Q$ are posets, then an injection $f : P \to Q$ is a \emph{weak embedding} if $f(x) \leq f(y)$ when $x \leq y$; we say that $f(P)$ is a \emph{copy} of $P$ in $Q$ and say that $Q$ is \emph{$P$-free} if there is no copy of $P$ in $Q$.
An injection $f : P \to Q$ is a \emph{strong embedding} if $f(x) \leq f(y)$ if and only if $x \leq y$; we say that $f(P)$ is an \emph{induced copy} of $P$ in $Q$.

A $k$-uniform \emph{partially-ordered hypergraph}, or \emph{pograph}, is a $k$-uniform hypergraph $H$ and a relation $\leq$ such that $(V(H),\leq)$ is a poset and every edge in the edge set $E(H)$ is a $k$-chain in $(V(H),\leq)$; note that it is not necessary that every $k$-chain be an edge.
If $H$ and $G$ are $k$-uniform pographs on posets $P$ and $Q$, then $G$ \emph{contains} a copy of $H$ if there is a weak embedding $f : P \to Q$ such that the graph $f(H)$ is a subgraph of $G$.
Let $\P = \{ P_n : n \geq 1\}$ be a family of posets such that $P_n\subseteq P_{n+1}$ for each $n$ and let $H_1,\dots,H_t$ be $k$-uniform pographs.
The \emph{partially-ordered Ramsey number} $\PR{\P}^k(H_1,\dots,H_t)$ is the minimum $N$ such that every $t$-coloring of the $k$-chains of $P_N$ contains a copy of $H_i$ in color $i$ for some $i \in \{1,\dots,t\}$.

The pographs $H_1,\dots,H_t$ are contained within ordered hypergraphs $G_1,\dots,G_t$ by extending the partial order to a total order; if $P_n$ contains a chain of size $\OR^k(G_1,\dots,G_t)$, then $\PR{\P}^k(H_1,\dots,H_t)\leq n$.
Thus, partially-ordered Ramsey numbers exist whenever the family $\P$ has unbounded height.
This is not a requirement, and we discuss several interesting poset families and their relations to other Ramsey numbers in Section~\ref{sec:relations}.
For the majority of this paper, we will focus on two natural poset families and use special notation to describe their Ramsey numbers.
Let $H_1,\dots,H_t$ be $k$-uniform pographs.
\begin{enumerate}
\item Let $C_n$ be a chain of $n$ elements, $\C = \{C_n : n \geq 1\}$, and define the \emph{chain Ramsey number} $\CR^k(H_1,\dots,H_t) = \PR{\C}^k(H_1,\dots,H_t)$.
\item Let $B_n$ be the Boolean lattice of subsets of $\{1,\dots,n\}$, $\B = \{ B_n : n \geq 1\}$, and define the \emph{Boolean Ramsey number} $\BR^k(H_1,\dots,H_t) = \PR{\B}^k(H_1,\dots,H_t)$.
\end{enumerate}
When $H_1 = \cdots = H_t = H$, we shorten our notation to $\CR_t^k(H) = \CR^k(H_1,\dots,H_t)$ and $\BR_t^k(H) = \BR^k(H_1,\dots,H_t)$.
The 2-uniform chain Ramsey numbers are a slight generalization of both ordered Ramsey numbers (if $H_1,\dots,H_t$ are totally ordered) and the \emph{directed Ramsey numbers\footnote{In \cite{CP02} these are called ordered Ramsey numbers.
See \cite{MSW15} for a detailed discussion about the distinction.}} defined by Choudum and Ponnusammy~\cite{CP02}, which consider coloring the edges of the transitive tournament to avoid monochromatic copies of certain directed acyclic graphs.

We mainly focus on 1- and 2-uniform Boolean Ramsey numbers, generalizing to other families only when the proof method is identical.
The 2-uniform Boolean Ramsey numbers are an interesting generalization of 2-uniform ordered Ramsey numbers, and we discuss them in Section~\ref{sec:2unif}.

Ramsey theory on posets was initiated by Ne{\v{s}}et{\v{r}}il and R\"odl \cite{nevsetvril1984combinatorial} focusing on induced copies of posets.
Many results~\cite{duffus1991fibres,kierstead1987ramsey,mccolm1991ramseyian,trotter1999ramsey} continue this perspective.
Recently, others~\cite{axenovich2015boolean,GRS99,JLM13} focused specifically on the 1-uniform Boolean Ramsey problem finding induced copies of posets inside the Boolean lattice. Our motivation for studying the non-induced situation stems from generalizing the notion of ordered Ramsey numbers, but several of the techniques used to bound induced Ramsey numbers apply in our situation.
We discuss 1-uniform Boolean Ramsey numbers in Section~\ref{sec:1unif}.
There is little interest in 1-uniform chain Ramsey numbers of graphs as they can be determined by basic application of the pigeonhole principle.
The 1-uniform Boolean Ramsey numbers relate to the very active area of 1-uniform Tur\'an problems in the Boolean lattice~\cite{DK07,DKS05,GL13,GL15,GL09,GMT14,KMY13}.
This area dates back to Sperner~\cite{Sperner28} who showed that the largest family of $B_n$ that does not contain a comparable pair has size ${n\choose\lfloor n/2\rfloor}$. 
These problems ask for the largest collection of elements in the Boolean lattice whose induced subposet does not contain a copy of a specific poset $P$.

Finding an exact value of a Boolean Ramsey number is very difficult.
We discuss computational methods to find small Boolean Ramsey numbers in Section~\ref{sec:computation}.

\subsection{Notation and Common Posets}

We follow standard notation from \cite{West96}. 
For integers $m\leq n$, we let $[n]=\{1,\dots,n\}$ and $[m,n]=\{m,m+1,\dots,n-1,n\}$. For a set $X$ and integer $d$, we let ${X\choose d}$ denote the set of all $d$-element subsets of $X$.
We use $\lg n=\log_2 n$ for shorthand.

For a poset $P$ and an element $x\in P$, we use $\down(x)=\{y\in P:y\leq x\}$ and $\up(x)=\{y\in P:x\leq y\}$, called the \emph{down-set of $x$} and the \emph{up-set of $x$} respectively. 
For a $t$-coloring $c$ of the 2-chains in $P$, an element $x \in P$, and a color $i \in [t]$, we define the \emph{$i$-colored down-set of $x$}, denoted $\down_i(x)$, to be the elements $y < x$ such that $c(yx) = i$; similarly the \emph{$i$-colored up-set of $x$}, denoted $\up_i(x)$, is the set of elements $y > x$ such that $c(xy) = i$.
The \emph{height} of a poset $P$, denoted $h(P)$, is the maximum size of a chain in $P$.

When we discuss 1-uniform pographs, we define only the poset and assume the set of ``edges'' is the same as the set of elements. 
In the case of 2-uniform pographs, we have two natural options for the edge set.
For a poset $(P,\leq)$, the \emph{comparability graph} is the pograph with vertex set $P$ and an edge $uv$ if and only if $u < v$.
The \emph{Hasse diagram} of $(P,\leq)$ is the pograph with vertex set $P$ and an edge $uv$ if and only if $u < v$ and there does not exist an element $w$ such that $u < w$ and $w < v$; such pairs $uv$ are \emph{cover relations}.
When we draw a pograph, adjacent vertices are comparable with the comparison ordered by height. 
We will focus mainly on a few natural 2-uniform pographs.


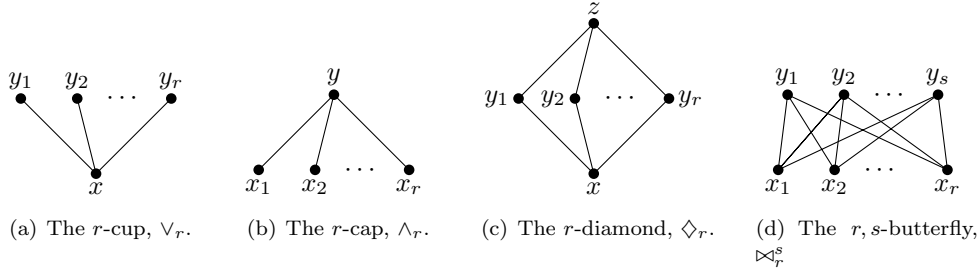
\begin{figure}[tp]
\centering
\subfigure[][{The $r$-cup, $\vee_r$.}]{\label{fig:cup}\begin{tikzpicture}[vtx/.style={shape=coordinate}]
\node[vtx,label=above:$y_1$] (y1) at (-1,1){};
\node[vtx,label=above:$y_2$] (y2) at (-.25,1) {};
\node[vtx,label=center:$\cdots$] (ydots) at (.375,1) {};
\node[vtx,label=above:$y_r$] (yr) at (1,1) {};
\node[vtx,label=below:$x$] (x) at (0,0){};

\draw (y1)--(x)--(y2)  (x)--(yr);

\foreach \p in {x,y1,y2,yr}
{
	\fill (\p) circle (2pt);
}
\end{tikzpicture} }
\quad
\subfigure[][\label{fig:cap}The $r$-cap, $\wedge_r$.]{\begin{tikzpicture}[vtx/.style={shape=coordinate}]
\node[vtx,label=above:$y$] (1) at (0,1) {};
\node[vtx,label=below:$x_1$] (2) at (-1,0){};
\node[vtx,label=below:$x_2$] (3) at (-.25,0) {};
\node[vtx,label=center:$\cdots$] (5) at (.375,0) {};
\node[vtx,label=below:$x_r$] (6) at (1,0) {};

\draw (2)--(1)--(3)  (1)--(6);

\foreach \p in {1,2,3,6}
{
	\fill (\p) circle (2pt);
}
\end{tikzpicture} }
\quad
\subfigure[][\label{fig:diamond}The $r$-diamond, $\diamond_r$.]{\begin{tikzpicture}[vtx/.style={shape=coordinate}]
\node[vtx,label=below:$x$] (1) at (0,0) {};
\node[vtx,label=left:$y_1$] (2) at (-1,1){};
\node[vtx,label=left:$y_2$] (3) at (-.25,1) {};
\node[vtx,label=center:$\cdots$] (5) at (.375,1) {};
\node[vtx,label=right:$y_r$] (6) at (1,1) {};
\node[vtx,label=above:$z$] (7) at (0,2) {};

\draw (7)--(2)--(1)--(3)--(7)--(6)--(1);

\foreach \p in {1,2,3,6,7}
{
	\fill (\p) circle (2pt);
}
\end{tikzpicture} }
\quad
\subfigure[][\label{fig:butterfly}The $r,s$-butterfly, $\bowtie_r^s$]{\begin{tikzpicture}[vtx/.style={shape=coordinate}]
\node[vtx,label=above:$y_1$] (y1) at (-1,1){};
\node[vtx,label=above:$y_2$] (y2) at (-.25,1) {};
\node[vtx,label=center:$\cdots$] (ydots) at (.375,1) {};
\node[vtx,label=above:$y_s$] (ys) at (1,1) {};
\node[vtx,label=below:$x_1$] (x1) at (-1.125,0){};
\node[vtx,label=below:$x_2$] (x2) at (-.375,0) {};
\node[vtx,label=center:$\cdots$] (xdots) at (.25,0) {};
\node[vtx,label=below:$x_r$] (xr) at (1.125,0) {};

\draw (y1)--(x1)--(y2)--(x2)--(ys)--(xr)--(y2)--(x1);
\draw (ys)--(x1) (y1)--(x2) (xr)--(y1);

\foreach \p in {x1,x2,xr,y1,y2,ys}
{
	\fill (\p) circle (2pt);
}
\end{tikzpicture} }
\caption{The cup, cap, diamond, and butterfly pographs, respectively.\label{fig:cupcapdiamond}}
\end{figure}

\begin{figure}[tp]
\centering
\subfigure[][\label{fig:diamond2}The diamond, $\diamond=\diamond_2$.]{\hspace{0.1in}\begin{tikzpicture}[vtx/.style={shape=coordinate}]
\node[vtx,label=below:$x$] (1) at (0,0) {};
\node[vtx,label=left:$y_1$] (2) at (-1,1){};
\node[vtx,label=right:$y_2$] (6) at (1,1) {};
\node[vtx,label=above:$z$] (7) at (0,2) {};

\draw (7)--(2)--(1)--(6)--(7);

\foreach \p in {1,2,6,7}
{
	\fill (\p) circle (2pt);
}
\end{tikzpicture} \hspace{0.1in}}
\quad
\subfigure[][\label{fig:boolean2}The 2-dimensional Boolean lattice, $B_2$.]{\begin{tikzpicture}[vtx/.style={shape=coordinate}]
\node[vtx,label=below:$x$] (1) at (0,0) {};
\node[vtx,label=left:$y_1$] (2) at (-1,1){};
\node[vtx,label=right:$y_2$] (6) at (1,1) {};
\node[vtx,label=above:$z$] (7) at (0,2) {};

\draw (7)--(2)--(1)--(6)--(7)--(1);

\foreach \p in {1,2,6,7}
{
	\fill (\p) circle (2pt);
}
\end{tikzpicture} }
\quad
\subfigure[][\label{fig:butterfly2}The butterfly, $\bowtie = \bowtie_2^2$]{\hspace{0.33in}\begin{tikzpicture}[vtx/.style={shape=coordinate}]
\node[vtx,label=above:$y_1$] (y1) at (-0.5,1){};
\node[vtx,label=above:$y_2$] (ys) at (0.5,1) {};
\node[vtx,label=below:$x_1$] (x1) at (-0.5,0){};
\node[vtx,label=below:$x_2$] (xr) at (0.5,0) {};

\draw (y1)--(x1)--(ys)--(xr)--(y1);

\foreach \p in {x1,xr,y1,ys}
{
	\fill (\p) circle (2pt);
}
\end{tikzpicture} \hspace{0.33in}}
\caption{The 2-diamond, $B_2$, and $\bowtie$.\label{fig:diamondbool}}
\end{figure}
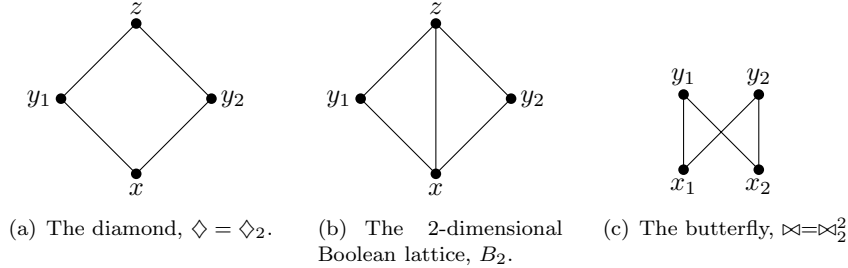

\begin{figure}[tp]
\centering
\subfigure[][\label{fig:matching}The $r$-matching, $M_r$.]{\hspace{0.1in}\begin{tikzpicture}[vtx/.style={shape=coordinate}]
\node[vtx,label=above:$y_1$] (y1) at (-1,1){};
\node[vtx,label=above:$y_2$] (y2) at (-.25,1) {};
\node[vtx,label=center:$\cdots$] (ydots) at (.375,0.5) {};
\node[vtx,label=above:$y_r$] (yr) at (1,1) {};
\node[vtx,label=below:$x_1$] (x1) at (-1,0){};
\node[vtx,label=below:$x_2$] (x2) at (-.25,0) {};
\node[vtx,label=below:$x_r$] (xr) at (1,0) {};

\draw (y1)--(x1)  (y2)--(x2) (yr)--(xr);

\foreach \p in {x1,x2,xr,y1,y2,yr}
{
	\fill (\p) circle (2pt);
}
\end{tikzpicture} \hspace{0.1in}}
\quad
\subfigure[][\label{fig:crown}The $r$-crown, $W_r$.]{\hspace{0.1in}\begin{tikzpicture}[vtx/.style={shape=coordinate}]
\node[vtx,label=above:$y_1$] (y1) at (-1,1){};
\node[vtx,label=above:$y_2$] (y2) at (-.25,1) {};
\node[vtx,label=center:$\cdots$] (ydots) at (.375,0.5) {};
\node[vtx,label=above:$y_r$] (yr) at (1,1) {};
\node[vtx,label=below:$x_1$] (x1) at (-1,0){};
\node[vtx,label=below:$x_2$] (x2) at (-.25,0) {};
\node[vtx,label=below:$x_r$] (xr) at (1,0) {};

\draw (y1)--(x1)--(y2)--(x2)--(0,0.33) (0.75,0.66)--(yr)--(xr)--(y1);

\foreach \p in {x1,x2,xr,y1,y2,yr}
{
	\fill (\p) circle (2pt);
}
\end{tikzpicture} \hspace{0.1in}}
\caption{The matching and the crown.\label{fig:matchingcrown}}
\end{figure}
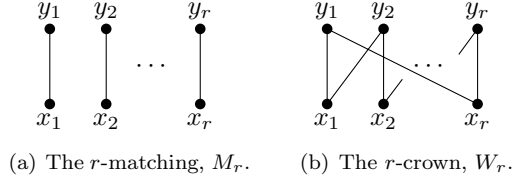

\begin{itemize}
\item The \emph{$n$-chain}, denoted $C_n$, is the Hasse diagram of $n$ totally-ordered elements.
\item The \emph{$n$-dimensional Boolean lattice}, denoted $B_n$, is the comparability graph of subsets of $[n]$ ordered by subset inclusion. See Figure~\ref{fig:boolean2} for a diagram of $B_2$.
\item The \emph{$r$-cup}, denoted $\vee_r$, is the comparability graph of the poset with elements $\{x,y_1,\dots,y_r\}$ where $x\leq y_i$ for all $i$ (see Figure~\ref{fig:cup}).
\item The \emph{$r$-cap}, denoted $\wedge_r$, is the comparability graph of the poset with elements $\{y,x_1,\dots,x_r\}$ where $x_i\leq y$ for all $i$ (see Figure~\ref{fig:cap}).
\item The \emph{$r$-diamond}, denoted $\diamond_r$, is the Hasse diagram of the poset with elements $\{x, y_1,\dots,y_r,z\}$ where $x \leq y_i \leq z$ for all $i$ (see Figures~\ref{fig:diamond} and \ref{fig:diamond2}). 
\item The \emph{$r,s$-butterfly}, denoted $\bowtie_r^s$, is the comparability graph of the poset with elements $\{x_1,\dots,x_r\} \cup \{y_1,\dots,y_s\}$ where $x_i \leq y_j$ for all $i$ and $j$; we use $\bowtie$ to denote $\bowtie_2^2$ (see Figures~\ref{fig:butterfly} and \ref{fig:butterfly2}).
\item The \emph{matching of size $n$}, denoted $M_n$, is the comparability graph of the poset with elements $\{x_1,\dots,x_n\} \cup \{y_1,\dots,y_n\}$ where $x_i\leq y_i$ for all $i$ (Figure \ref{fig:matching}).
\item The \emph{crown graph of order $n$}, denoted $W_n$, is the comparability graph of the poset with elements $\{x_1,\dots,x_n\} \cup \{y_1,\dots,y_n\}$ where $x_i\leq y_i$ and $x_i\leq y_{i+1\!\! \pmod{n}}$ for all $i$ (Figure \ref{fig:crown}).
\end{itemize}

Note the difference between the 2-diamond $\diamond_2$ and the 2-dimensional Boolean lattice $B_2$.
Both pographs are defined for the same poset, but $\diamond_2$ is the Hasse diagram and hence has one fewer edge than the comparability graph in $B_2$.
This distinction leads to different values of 2-uniform Boolean Ramsey numbers; see Section~\ref{sec:computation}.

Even though we defined these symbols in terms of 2-uniform pographs, we will often use the same symbol to denote the 1-uniform pograph, which is simply the underlying poset. Whether we are discussing the 1- or 2-uniform case will always be clear from context.
%

\section{1-Uniform Boolean Ramsey Numbers}\label{sec:1unif}

For a poset $P$, define $\levels(P)$ to be the maximum $m$ such that, for all $n$, the union of the middle $m$ levels of $B_n$ does not contain a copy of $P$.
The parameter $\levels(P)$ is very common in the study of Tur\'an-type problems in posets.

\begin{prop}\label{prop:1ulow}
Let $P_1,\dots,P_t$ be posets.
If $M$ is the least integer such that $P_i\subseteq B_M$ for all $i$, then
\[
\max\left\{M,\sum_{i=1}^t \levels(P_i)\right\}\leq\BR^1(P_1,\dots,P_t)\leq\sum_{i=1}^t(|P_i|-1).
\]
\end{prop}

\begin{proof}
The upper bound follows from the fact that $B_n$ contains a chain of length $n+1$ and that $P_i\subseteq C_{|P_i|}$.
For the lower bound, let $n=\sum_{i=1}^t \levels(P_i)-1$.
For $v\in B_n$ with $|v|\in\left[\sum_{j=1}^{i-1}\levels(P_j),\sum_{j=1}^i \levels(P_j)-1\right]$, let $c(v)=i$.
Thus $c^{-1}(i)$ is the union of $\levels(P_i)$ consecutive levels of $B_n$, so $c$ avoids copies of $P_i$ in color $i$ for all $i$.
\end{proof}

Later in this section, we will demonstrate situations in which the lower bound in Proposition \ref{prop:1ulow} is not tight, but we believe that the lower bound is much closer to the correct answer.
To this end, we believe that the upper bound in Proposition \ref{prop:1ulow} is far from tight in most cases and conjecture that the upper bound is tight \emph{only} when each $P_i$ is a chain (and hence $\levels(P_i) = |P_i|-1$).

The remainder of this section sets out to determine the $1$-uniform Boolean Ramsey numbers of various posets.

\begin{theorem}\label{thm:boolchain}
For positive integers $n_1,\dots,n_t$, $\BR^1(B_{n_1},C_{n_2},\dots,C_{n_t})=n_1+\sum_{i=2}^t(n_i-1)$.
\end{theorem}

\begin{proof}
The lower bound follows from Proposition \ref{prop:1ulow}, so we need only show the upper bound.

We first prove that $\BR^1(B_n,C_m)\leq n+m-1$ by induction on $m$.
For $m=1$, the result is immediate as any use of color $2$ creates a chain of order 1.
Suppose that $m\geq 2$ and let $N=n+m-1$.
Let $c$ be any $2$-coloring of $B_N$ and suppose that $c$ avoids copies of $C_m$ in color $2$; we will show that $c$ must admit a copy of $B_n$ in color $1$.
Let $L$ be the family of subsets of $[N-1]$.
As $L$ is a copy of $B_{N-1}$, the induction hypothesis states that $c$ restricted to $L$ must admit either a copy of $B_{n}$ in color $1$ or a copy of $C_{m-1}$ in color $2$.
If the former holds, then we are done.
Otherwise, $c$ restricted to $L$ admits a copy of $C_{m-1}$ in color $2$.
Suppose that $X_1,\dots,X_s$ are the copies of $C_{m-1}$ in color $2$ contained in $L$.
Because $c$ avoids copies of $C_m$ in color $2$, we see that the elements in $\bigcup_{i=1}^s\left(\up( \max X_i)\setminus\max X_i\right)$ all have color 1.
Let $U=\bigcup_{i=1}^s\up(\max X_i)\cap L$ and let $U'=\{Y\cup\{N\}:Y\in U\}$.
Notice that $U'\subseteq\bigcup_{i=1}^s\left(\up( \max X_i)\setminus\max X_i\right)$, so $U'$ contains only elements of color $1$.
Furthermore, it is easily seen that $B_{N-1}$ embeds into $(L\setminus U)\cup U'$ as $U\cong U'$ and $U'$ is an up-set.
However, $c$ restricted to $(L\setminus U)\cup U'$ does not contain any copies of $C_{m-1}$ in color $2$, so by the induction hypothesis, it must admit a copy of $B_n$ in color $1$ as needed.
We conclude that $\BR^1(B_{n},C_m)\leq N$.

Now that we have proved that $\BR^1(B_n,C_m)=n+m-1$, the $t$-color version follows quickly.
Let $m = 1 + \sum_{i=2}^t (n_i-1)$ and $N = n_1 + m - 1$..
From the 2-color case, $\BR^1(B_{n_1}, C_m) = N = n_1 + \sum_{i=2}^t(n_i-1)$.
Thus, if $c$ is a $t$-coloring of $B_N$, then either $c$ admits a copy of $B_{n_1}$ in color 1 or there exists a copy of $C_m$ where all elements have color in $\{2,\dots,t\}$.
If there is a copy of $B_{n_1}$ in color 1, then we are done.
If not, since $m = 1 + \sum_{i=2}^t (n_i-1)$, there exists a chain of size $n_i$ within the copy of $C_m$ that has color $i$ for some $i \in \{2,\dots,t\}$ by the pigeonhole principle.
\end{proof}

The proof of the upper bound in Theorem \ref{thm:boolchain} actually shows the following statement for general posets.

\begin{cor}
For posets $P_1,\dots,P_t$ and positive integers $n_1,\dots,n_s$, 
\[
\BR^1(P_1,\dots,P_t,C_{n_1},\dots,C_{n_s})\leq\BR^1(P_1,\dots,P_t)+\sum_{i=1}^s(n_i-1).
\]
\end{cor}

Axenovich and Walzer~\cite{axenovich2015boolean} consider the induced version of Boolean Ramsey number.
Their upper bounds imply upper bounds for our version, so we summarize some of their results here.

\begin{theorem}[Axenovich and Walzer~{\cite[Theorem 1]{axenovich2015boolean}}]
Let $r \geq s \geq 1$.
\begin{itemize}
\item $\BR^1(B_r,B_s) \leq rs+r+s$.
\item $\BR^1(B_r,B_2) \leq 2r+2$.
\item $\BR^1(B_3,B_3) \leq 8$.
\end{itemize}
\end{theorem}

Note that the lower bounds established by Axenovich and Walzer do not carry over to our case, although we will later prove that $\BR^1(B_3,B_3)\geq 7$ as is the case in the induced version.

A common tool in studying Tur\'an-type questions in posets is known as the \emph{Lubell function}. For a family $\F \subseteq B_n$, the \emph{Lubell function of $\F$} is defined as
\[
\lu_n(\F)=\sum_{F\in\F}{n\choose |F|}^{-1}.
\]
The Lubell function of $\F$ can be interpreted as the average size of $|\F\cap\C|$ where $\C$ is a full chain in $B_n$. 
An alternate interpretation is that $\lu_n(\F)$ is the expected number of elements of $\F$ that are visited by a random walk from the empty set to the full set along the Hasse diagram of $B_n$.
Using either interpretation, it is straightforward to observe that $|\F|\leq\lu_n(\F){n\choose\lfloor n/2\rfloor}$. 
It is due to this observation that Lubell functions of $P$-free families have received a great deal of attention, as bounds on the Lubell function help answer Tur\'an-type questions in the Boolean lattice. 
We apply Lubell functions to attain bounds on the $1$-uniform Boolean Ramsey number by calling upon linearity, i.e.\ if $\F\cap\G=\varnothing$, then $\lu_n(\F\cup\G)=\lu_n(\F)+\lu_n(\G)$.

For a poset $P$, let $L_n(P)$ be the maximum value $\lu_n(\F)$ among families $\F \subseteq B_n$ such that $\F$ is $P$-free.
The following result has been used implicitly by many authors such as Axenovich and Walzer~{\cite[Theorem 6]{axenovich2015boolean}} and Johnston, Lu and Milans~{\cite[Theorem 3]{JLM13}}. Although the result is straightforward, we provide the explicit statement in the case of 1-uniform Boolean Ramsey numbers along with a proof for completeness.

\begin{theorem}\label{lem:lubellbound}
If $P_1,\dots,P_t$ are posets and $\sum_{i=1}^t L_n(P_i) < n+1$ for some integer $n$, then $\BR^1(P_1,\dots,P_t) \leq n$.
\end{theorem}

\begin{proof}
Let $c$ be a $t$-coloring of $B_n$ and suppose that it avoids copies of $P_i$ in color $i$ for all $i$ and define $\F_i=c^{-1}(i)$. As $\F_i$ is $P_i$-free, we know that $\lu_n(\F_i)\leq L_n(P_i)$ for all $i$. Therefore, as the $\F_i$'s are disjoint,
\[
\sum_{i=1}^t L_n(P_i)\geq\sum_{i=1}^t \lu_n(\F_i)=\lu_n\left(\bigcup_{i=1}^t\F_i\right)=\lu_n(B_n)=n+1.
\]
As such, if $\sum_{i=1}^t L_n(P_i)<n+1$ for some $n$, then $\BR^1(P_1,\dots,P_t)\leq n$.
\end{proof}

%
%

%
%

Griggs and Li~\cite{GL15} define a poset $P$ to be \emph{uniformly Lubell-bounded}, or \emph{uniformly L-bounded}, if $L_n(P)\leq \levels(P)$ for all $n$. 
By a direct application of Theorem \ref{lem:lubellbound}, we find that the lower bound given in Proposition \ref{prop:1ulow} is tight when regarding uniformly L-bounded posets.

\begin{prop}\label{prop:ulbound}
If $P_1,\dots,P_t$ are uniformly L-bounded posets, then
\[
\BR^1(P_1,\dots,P_t)=\sum_{i=1}^t \levels(P_i).
\]
\end{prop}

\begin{cor}
For a positive integer $r$, let $m=\lceil\lg(r+2)\rceil$. 
\begin{enumerate}
\item If $r\in\left[2^{m-1}-1,2^m-{m\choose\lfloor m/2\rfloor}-1\right]$, then $\BR^1_t(\diamond_r)=tm$.
\item If $r\in\left[2^m-{m\choose\lfloor m/2\rfloor},2^m-2\right]$, then $tm\leq\BR^1_t(\diamond_r)\leq t(m+1)-\left\lceil t(2^m-r-1)/{m\choose\lfloor m/2\rfloor}\right\rceil$.
\end{enumerate}
\end{cor}

\begin{proof}
It is well-known that $\levels(\diamond_r)=m$, so $\BR^1_t(\diamond_r)\geq tm$ by Proposition~\ref{prop:1ulow}.

%
Griggs, Li and Lu~{\cite[Theorem 2.5]{GLL12}} showed that if $r\in\left[2^{m-1}-1,2^m-{m\choose\lfloor m/2\rfloor}-1\right]$, then $L_n(\diamond_r)\leq m$ and if $r\in\left[2^m-{m\choose\lfloor m/2\rfloor},2^m-2\right]$, then $L_n(\diamond_r)\leq m+1-(2^m-r-1)/{m\choose\lfloor m/2\rfloor}$.

As such, for $r\in\left[2^{m-1}-1,2^m-{m\choose\lfloor m/2\rfloor}-1\right]$, $\diamond_r$ is uniformly L-bounded, so $\BR^1_t(\diamond_r)=tm$ by Proposition~\ref{prop:ulbound}. 

For $r\in\left[2^m-{m\choose\lfloor m/2\rfloor},2^m-2\right]$, let $M=(2^m-r-1)/{m\choose\lfloor m/2\rfloor}$. For any $x\in\mathbb{R}$, $\lceil x\rceil<x+1$, so
\[
t\cdot L_n(\diamond_r)\leq t(m+1)-tM<t(m+1)-\lceil tM\rceil+1,
\]
for each $n$. Therefore, by Proposition~\ref{lem:lubellbound}, we have $\BR^1_t(\diamond_r)\leq t(m+1)-\lceil tM\rceil$ .
\end{proof}

Up until this point, we have mostly considered cases in which the lower bound in Proposition \ref{prop:1ulow} is tight. 
This, however, is not the case in general.
To show this, we consider the \emph{butterfly poset}.
The butterfly poset is special as De Bonis, Katona, and Swanepool \cite{DKS05} determined the largest $\bowtie$-free family in $B_n$ to be \emph{exactly} the middle two levels \emph{for all} $n$, while most other results in this direction are necessarily asymptotic. This is especially interesting as $L_n(\bowtie)=3$ for all $n$, which is witnessed by any family consisting of a level of $B_n$ along with $\varnothing$ and $[n]$. As such, Theorem \ref{lem:lubellbound} implies that $\BR_t(\bowtie)\leq 3t$, but this is not tight.
In order to determine $\BR_t^1(\bowtie)$, we require a more careful use of the idea in Theorem \ref{lem:lubellbound}. 
For a poset $P$, define
\[
L_n'(P)=\max\left\{\lu_n(\F):\F\subseteq B_n\setminus\{[n],\varnothing\},\text{$\F$ is $P$-free}\right\},
\]
This new value $L_n'(P)$ is the maximum Lubell value of a $P$-free family that does not contain either the maximal or minimal element.

\begin{prop}
For $t\geq 1$, $\BR_t^1(\bowtie)=2t+1$.
\end{prop}

\begin{proof}
\textit{Lower bound.} Let $c$ be a $t$-coloring of $B_{2t}$ defined as follows.
For $i\in[t-1]$, if $|x|\in\{2i,2i+1\}$, let $c(x)=i$, and if $|x|\in\{0,1,2t\}$, let $c(x)=t$.
As $\levels(\bowtie)=2$, we see that $c$ avoids copies of $\bowtie$ in colors $1,\dots,t-1$.
Further, it is easy to check that $\bowtie$ does not appear in color $t$, so $\BR_t^1(\bowtie)>2t$.

\textit{Upper bound.} As shown by Griggs and Li~{\cite[Theorem 5.1]{GL13}}, $L_n'(\bowtie)=2$. 
We begin by showing that for any $n$, $L_n'(\vee_2)<2$. 
Suppose that $L_n'(\vee_2)\geq 2$ and let $\F\subseteq B_n\setminus\{[n],\varnothing\}$ be a $\vee_2$-free family with $\lu_n(\F)\geq 2$. 
As $\vee_2$ is contained in $C_3$, we observe that no chain can intersect $3$ elements of $\F$, so $\lu_n(\F)=2$ and every full chain in $B_n$ must intersect \emph{exactly} $2$ elements of $\F$. 
Let $\mathcal{C}$ be any full chain in $B_n$ and suppose that $\mathcal{C}\cap\F=\{F_1,F_2\}$ with $F_1\subset F_2$. 
Now choose $\mathcal{C}'$ to be a full chain that agrees with $\mathcal{C}$ through $F_1$ and avoids $F_2$ (note that $\mathcal{C}'$ can be found as $F_2\neq[n]$).
Therefore, $\mathcal{C}'\cap\F=\{F_1,F_3\}$ for some $F_3\neq F_2$. 
As $\mathcal{C}$ and $\mathcal{C}'$ agree through $F_1$, it must be that $F_1\subset F_3$, so $F_1F_2F_3$ forms a copy of $\vee_2$; a contradiction. 
Thus, $L_n'(\vee_2)<2$ for every $n$.

Now let $c$ be any $t$-coloring of $B_{2t+1}$, and, without loss of generality, suppose that $c(\varnothing)=1$.
Notice that if $c$ restricted to $B_{2t+1}\setminus\{\varnothing\}$ admits a copy of $\vee_2$ in color $1$, then $c$ admits a copy of $\bowtie$ in color $1$.
For $i\in[t]$, define $\F_i=c^{-1}(i)\setminus\{[2t+1],\varnothing\}$ and suppose that it were the case that $\F_1$ is $\vee_2$-free and $\F_i$ is $\bowtie$-free for all $i\in\{2,\dots,t\}$. Thus, $\lu_{2t+1}(\F_1)\leq L_{2t+1}'(\vee_2)<2$ and $\lu_{2t+1}(\F_i)\leq L_{2t+1}'(\bowtie)\leq 2$ for all $i\in\{2,\dots,t\}$. Therefore, by linearity,
\[
2t>L_{2t+1}'(\vee_2)+\sum_{i=2}^t L_{2t+1}'(\bowtie)\geq\sum_{i=1}^t\lu_{2t+1}(\F_i)=\lu_{2t+1}\left(\bigcup_{i=1}^t\F_i\right)=\lu_{2t+1}(B_{2t+1}\setminus\{[2t+1],\varnothing\})=2t;
\]
a contradiction. Thus, $\BR^1_t(\bowtie)=2t+1$.
\end{proof}

Notice that the proof of the lower bound extends to show that for any posets $P_2,\dots,P_t$,
\[
\BR^1(\bowtie,P_2,\dots,P_t)\geq 3+\sum_{i=2}^t\levels(P_i),
\]
even though $\levels(\bowtie)=2$. We now present another case in which the lower bound in Proposition \ref{prop:1ulow} is not tight. 
In order to do so, we require the use of the Lov\'asz Local Lemma, which we briefly state.

\begin{theorem}[Lov\'asz Local Lemma]\label{thm:local}
Let $A_1,\dots,A_k$ be a collection of events such that for all $i$, $\Pr[A_i]\leq p$ and $A_i$ is independent of all but at most $D$ other events. If $ep(D+1)\leq 1$, where $e$ is the base of the natural logarithm, then there is a nonzero probability that none of the events occur. 
\end{theorem}

\begin{theorem}\label{thm:lowerboolean}
$\BR^1_2(B_d)\geq 2d+1$ for $3\leq d\leq 8$ and $d\geq 13$.
\end{theorem}

\begin{proof}
%
%
%
%
We will construct a $2$-coloring of $B_{2d}$ that does not contain a monochromatic copy of $B_d$. In order to do so, we will rely on the existence of a family $R_d\subseteq{[2d]\choose d}$ with the following two properties:
\begin{enumerate}
\item For each $S\in{[2d]\choose d}$, $R_d$ contains exactly one of $S$ and $S^C$.
\item For each $T\in{[2d]\choose d+1}$, $R_d$ contains at most $d-1$ subsets of $T$.
\end{enumerate}

For $3\leq d\leq 8$, we have verified that $R_d$ exists by solving a suitable integer program, so we only need find $R_d$ for $d\geq 13$.

Independently for a pair $\{S,S^C\}\subseteq{[2d]\choose d}$, uniformly select which of $S$ and $S^C$ is in $R_d$. Now, for $T\in{[2d]\choose d+1}$, let $A_T$ be the the event in which $R_d$ contains at least $d$ subsets of $T$. As $|T|=d+1$ and thus contains at most one of $S$ and $S^C$ for each $S\in{[2d]\choose d}$,
\[
\Pr[A_T]={d+1\over 2^{d+1}}+{1\over 2^{d+1}}={d+2\over 2^{d+1}}.
\]
Now we note that $A_T$ and $A_{T'}$ are independent unless either $T$ and $T'$ have a common $d$-subset or $T'$ contains the complement of a $d$-subset of $T$; in other words, as $|T|=|T'|=d+1$, $A_T$ and $A_{T'}$ are independent unless $|T\cap T'|\in\{2,d\}$. For a fixed $T$, there are precisely ${d+1\choose d}{d-1\choose 1}+{d+1\choose 2}={1\over 2}(d+1)(3d-2)$ choices of $T'$ for which this happens. As such, $A_T$ is independent of all but $D={1\over 2}(d+1)(3d-2)$ of the other $A_{T'}$'s. As $d\geq 13$, we have
\[
ep(D+1)=e{d+2\over 2^{d+1}}\left({1\over 2}(d+1)(3d-2)+1\right)\leq 1,
\]
so Theorem \ref{thm:local} implies that there is a positive probability that none of the $A_T$'s occur. In other words, for all $d\geq 13$, $R_d$ exists.

We construct a coloring $c$ of $B_{2d}$ as follows: for $x\in B_{2d}$, if $x\in R_d$ or $|x|\in\{0,1,\dots,d-2,d+1\}$, let $c(x)=1$ and otherwise let $c(x)=2$ (see Figure \ref{fig:B6coloring} for an example when $d=3$). By the first property of $R_d$, we observe that $c(x)=1$ if and only if $c(x^C)=2$, so it suffices to show that the family $\F=c^{-1}(1)$ does not contain a copy of $B_d$. Let $\F_i=\F\cap{[2d]\choose i}$. We first observe that $\F$ has height $d+1$, so if it were to contain a copy of $B_d$, it must be the case that the $i$th level of the copy of $B_d$ lies entirely in $\F_i$ for $i\in\{0,\dots,d-2\}$, the $(d-1)$st level lies entirely in $\F_d$ and the maximal element lies in $\F_{d+1}$. Thus, there must be some $x\in\F_{d+1}$ for which $|\down(x)\cap\F_d|\geq d$. However, as $\F_d=R_d$, this is not the case by the second property of $R_d$. Hence, $c$ does not admit a monochromatic copy of $B_d$.
\end{proof}

\begin{figure}[h]
\scalebox{.7}{\begin{tikzpicture}[vtx/.style={shape=coordinate}]
\node[vtx,label=left:$124$,circle=2pt,fill=black] (124) at (-5.25,0) {};
\node[vtx,label=left:$146$,circle=2pt,fill=black] (146) at (-2.25,0) {};
\node[vtx,label=left:$156$,circle=2pt,fill=black] (156) at (-.75,0) {};
\node[vtx,label=right:$256$,circle=2pt,fill=black] (256) at (3.75,0) {};
\node[vtx,label=right:$245$,circle=2pt,fill=black] (245) at (2.25,0) {};
\node[vtx,label=right:$236$,circle=2pt,fill=black] (236) at (.75,0) {};
\node[vtx,label=right:$345$,circle=2pt,fill=black] (345) at (5.25,0) {};
\node[vtx,label=right:$346$,circle=2pt,fill=black] (346) at (6.75,0) {};
\node[vtx,label=left:$135$,circle=2pt,fill=black] (135) at (-3.75,0) {};
\node[vtx,label=left:$123$,circle=2pt,fill=black] (123) at (-6.75,0) {};

\node[vtx,label=left:$1$,circle=2pt,fill=black] (1) at (-3.75,-3) {};
\node[vtx,label=left:$2$,circle=2pt,fill=black] (2) at (-2.25,-3) {};
\node[vtx,label=left:$3$,circle=2pt,fill=black] (3) at (-.75,-3) {};
\node[vtx,label=right:$4$,circle=2pt,fill=black] (4) at (.75,-3) {};
\node[vtx,label=right:$5$,circle=2pt,fill=black] (5) at (2.25,-3) {};
\node[vtx,label=right:$6$,circle=2pt,fill=black] (6) at (3.75,-3) {};

\node[vtx,label=above:$1234$,circle=2pt,fill=black] (1234) at (-7,3) {};
\node[vtx,label=above:$1235$,circle=2pt,fill=black] (1235) at (-6,3) {};
\node[vtx,label=above:$1236$,circle=2pt,fill=black] (1236) at (-5,3) {};
\node[vtx,label=above:$1245$,circle=2pt,fill=black] (1245) at (-4,3) {};
\node[vtx,label=above:$1246$,circle=2pt,fill=black] (1246) at (-3,3) {};
\node[vtx,label=above:$1256$,circle=2pt,fill=black] (1256) at (-2,3) {};
\node[vtx,label=above:$1345$,circle=2pt,fill=black] (1345) at (-1,3) {};
\node[vtx,label=above:$1346$,circle=2pt,fill=black] (1346) at (0,3) {};
\node[vtx,label=above:$1356$,circle=2pt,fill=black] (1356) at (1,3) {};
\node[vtx,label=above:$1456$,circle=2pt,fill=black] (1456) at (2,3) {};
\node[vtx,label=above:$2345$,circle=2pt,fill=black] (2345) at (3,3) {};
\node[vtx,label=above:$2346$,circle=2pt,fill=black] (2346) at (4,3) {};
\node[vtx,label=above:$2356$,circle=2pt,fill=black] (2356) at (5,3) {};
\node[vtx,label=above:$2456$,circle=2pt,fill=black] (2456) at (6,3) {};
\node[vtx,label=above:$3456$,circle=2pt,fill=black] (3456) at (7,3) {};

\node[vtx,label=below:$\varnothing$,circle=2pt,fill=black] (0) at (0,-4) {};

\node[label=left:${[6]\choose 0}$] (00) at (-8,-4) {};
\node[label=left:${[6]\choose 1}$] (11) at (-8,-3) {};
\node[label=left:$R_3$] (22) at (-8,0) {};
\node[label=left:${[6]\choose 4}$] (44) at (-8,3) {};

\draw (1)--(0)--(2)
(3)--(0)--(4)
(5)--(0)--(6)
(1)--(124)
(1)--(146)
(1)--(156)
(1)--(135)
(1)--(123)
(2)--(124)
(2)--(256)
(2)--(245)
(2)--(236)
(2)--(123)
(3)--(236)
(3)--(345)
(3)--(346)
(3)--(135)
(3)--(123)
(4)--(124)
(4)--(146)
(4)--(245)
(4)--(345)
(4)--(346)
(5)--(135)
(5)--(156)
(5)--(345)
(5)--(256)
(5)--(245)
(6)--(146)
(6)--(156)
(6)--(256)
(6)--(236)
(6)--(346)

(124)--(1234)
(124)--(1245)
(124)--(1246)
(146)--(1246)
(146)--(1346)
(146)--(1456)
(156)--(1256)
(156)--(1356)
(156)--(1456)
(256)--(1256)
(256)--(2356)
(256)--(2456)
(245)--(1245)
(245)--(2345)
(245)--(2456)
(236)--(1236)
(236)--(2346)
(236)--(2356)
(345)--(1345)
(345)--(2345)
(345)--(3456)
(346)--(1346)
(346)--(2346)
(346)--(3456)
(135)--(1235)
(135)--(1345)
(135)--(1356)
(123)--(1234)
(123)--(1235)
(123)--(1236)
;
\end{tikzpicture} }
\caption{A coloring of $B_6$ which avoids monochromatic copies of $B_3$. We only show those elements which receive color $1$; all other elements receive color $2$. \label{fig:B6coloring}}
\end{figure}
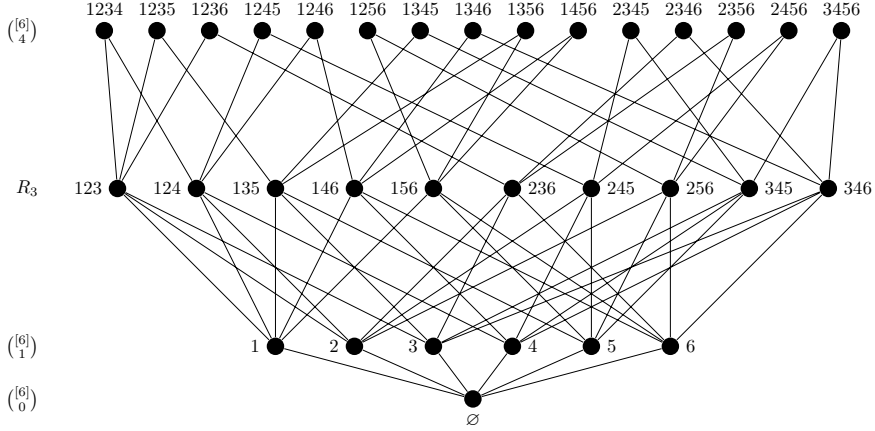

Unfortunately, computational limits prevent us from showing that $\BR^1_2(B_d)\geq 2d+1$ for all $d\geq 3$, but $R_d$ should certainly exist for $d\in\{9,10,11,12\}$. Further, even though Theorem \ref{thm:lowerboolean} only gives an improvement of $1$ over the trivial lower bound of $2d$, any improvement on the trivial lower bound in Proposition \ref{prop:1ulow} is of interest.

To investigate upper bounds on the Boolean Ramsey numbers of posets with size much larger than their height, we consider a structure that is \emph{wider} than a single chain and use that to consider an extension of Lubell functions.
We use an idea that Gr\'osz, Methuku and Tompkins~\cite{GMT14} used to approach the Tur\'an-type question.
Let $A \subseteq B$ and define the \emph{interval from $A$ to $B$}, denoted $[A,B]$, to be the collection of sets $C$ where $A \subseteq C \subseteq B$; we say the interval $[A, B]$ has \emph{height $m$} if $m = |B \setminus A|$. 
For a full chain $\A = (A_0,\dots,A_n)$ in $B_n$ where $\varnothing = A_0 \subset A_1 \subset \cdots\subset A_n = [n]$, define the \emph{$m$-interval chain} $\C_m(\A)$ as 
\[
\C_m(\A)=\bigcup_{i=0}^{n-m}\left[A_i,A_{i+m}\right].
\]
Gr\'osz, Methuku and Tompkins~\cite{GMT14} noted $|\C_m(\A)|=(n-m+2)2^{m-1}$ for all $m$-interval chains $\C_m(\A)$, which will be important.

For a family $\F \subseteq B_n$ and $1\leq m\leq n$, define the \emph{$m$-interval Lubell function of $\F$}, denoted $\lu_n^{(m)}(\F)$, as
\[
\lu_n^{(m)}(\F)=\frac{1}{n!}\sum_{\A}|\F\cap\C_m(\A)|
\]
where the sum is taken over all full chains $\A$. 
Observe that $\lu_n^{(1)}(\F)=\lu_n(\F)$. 
For a poset $P$ define $L_n^{(m)}(P) =\max\{\lu_n^{(m)}(\F):\F\subseteq B_n,\text{ $\F$ is $P$-free}\}$.

Due to the size of an $m$-interval chain, $\lu_n^{(m)}(B_n)=(n-m+2)2^{m-1}$. 
With this in mind, we arrive at a direct extension of Theorem \ref{lem:lubellbound}.

\begin{prop}\label{prop:chainlubell}
Let $P_1,\dots,P_t$ be posets. If $\sum_{i=1}^t L_n^{(m)}(P_i)<(n-m+2)2^{m-1}$ for some $m$ with $1\leq m\leq n$, then $\BR^1(P_1,\dots,P_t)\leq n$.
\end{prop}

In order to apply Proposition \ref{prop:chainlubell} to attain a general upper bound on 1-uniform Boolean Ramsey numbers, we will make use of the following result.

\begin{theorem}[Gr\'osz, Methuku and Tompkins~{\cite[Lemma 12]{GMT14}}]\label{thm:gmtbound}
For $m\geq 2$, if $P$ is a poset of height $h(P)$ and $\F$ is $P$-free, then for any $m$-interval chain $\C_m(\A)$,
\[
|\F\cap\C_m(\A)|\leq |P|-1+(h(P)-1)(3m-5)2^{m-2}.
\]
\end{theorem}

We now provide a general upper bound on the 1-uniform Boolean Ramsey number for posets whose sizes are large compared to their heights.

\begin{theorem}\label{thm:mlubell}
Let $P_1,\dots,P_t$ be posets, $S=\sum_{i=1}^t(|P_i|-1)$, and $H=\sum_{i=1}^t(h(P_i)-1)$.
\[
\BR^1(P_1,\dots,P_t) \leq  \left({3\over 2}H+1\right)\left(\lg\left({S\over H}\right)+1\right).
\]
\end{theorem}

\begin{proof}
If $n=\BR^1(P_1,\dots,P_t)-1$, then $\sum_{i=1}^t L_n^{(m)}(P_i)\geq(n-m+2)2^{m-1}$ whenever $1\leq m\leq n$  by Proposition \ref{prop:chainlubell}. 
Therefore, by Theorem \ref{thm:gmtbound},
\[
(n-m+2)2^{m-1}\leq\sum_{i=1}^t L_n^{(m)}(P_i)\leq\sum_{i=1}^t\left((|P_i|-1)+(h(P_i)-1)(3m-5)2^{m-2}\right)= S+(3m-5)2^{m-2}H
\]
for all $2\leq m\leq n$. As $S\geq H$, $\lg\left({S\over H}\right)\geq 0$, so set $m=\left\lfloor 2+\lg\left({S\over H}\right)\right\rfloor$. Hence,
\begin{align*}
n &\leq \left({3\over 2}H+1\right)m-{5\over 2}H+2^{1-m}S-2\\
&= \left({3\over 2}H+1\right)\left\lfloor 2+\lg\left({S\over H}\right)\right\rfloor-{5\over 2}H+2^{1-\left\lfloor 2+\lg\left({S\over H}\right)\right\rfloor}S-2\\
&\leq \left({3\over 2}H+1\right)\left(2+\lg\left({S\over H}\right)\right)-{5\over 2}H+2^{-\lg\left({S\over H}\right)}S-2\\
&= {3\over 2}H\lg\left({S\over H}\right)+{3\over 2}H+\lg\left({S\over H}\right) \\
&= \left({3\over 2}H+1\right)\left(\lg\left({S\over H}\right)+1\right)-1.\qedhere
\end{align*}
\end{proof}

A direct application of Theorem \ref{thm:mlubell} presents reasonable bounds on the Boolean Ramsey number of various poset families.

\begin{cor}\label{cor:bounds}
For positive integers $d,r_1,\dots,r_t,s_1,\dots,s_t$,
\begin{align*}
\BR_t^1(B_d) &\leq \left({3\over 2}dt+1\right)\left(\lg\left({2^d-1\over d}\right)+1\right)\leq 2d^2t,\\
\BR^1(\bowtie_{s_1}^{r_1},\dots,\bowtie_{s_t}^{r_t}) &\leq \left({3\over 2}t+1\right)\left(\lg\left({1\over t}\sum_{i=1}^t(r_i+s_i-1)\right)+1\right),\\
\BR^1(\vee_{r_1},\dots,\vee_{r_s},\wedge_{r_{s+1}},\dots,\wedge_{r_t})&\leq\left({3\over 2}t+1\right)\left(\lg\left({1\over t}\sum_{i=1}^t r_i\right)+1\right), \text{and}\\
\BR^1(\diamond_{r_1},\dots,\diamond_{r_t}) &\leq \left( 3t+1\right)\lg\left({1\over t}\sum_{i=1}^t(r_i+1)\right).\end{align*}
\end{cor}



Unfortunately, Theorem \ref{thm:mlubell} only allows us the show that $\BR^1_t(B_d)$ is at most quadratic in $d$ while we in fact suspect that it is linear.


\section{2-Uniform Boolean Ramsey Numbers}\label{sec:2unif}

We now focus on 2-uniform partially-ordered Ramsey numbers. 
Due to recent interest in ordered Ramsey numbers, we will also include results concerning chain Ramsey numbers.
We also will state our results in the $k$-uniform case when possible.

\begin{prop}\label{prop:chainboundbool}
Let $G_1,\dots,G_t$ be $k$-uniform pographs.
\[
\lg\CR^k(G_1,\dots,G_t)\leq\BR^k(G_1,\dots,G_t)\leq \CR^k(G_1,\dots,G_t)-1.
\]
\end{prop}

\begin{proof}
Let $N = \CR^k(G_1,\dots,G_t)$.
Observe that the chain $C_N$ is contained in the Boolean lattice $B_{N-1}$, so any $t$-coloring of the $k$-chains in $B_{N-1}$ contains a copy of $G_i$ in the color $i$ for some $i \in [t]$ and hence $\BR^k(G_1,\dots,G_t) \leq N-1$.

Let $n=\BR^k(G_1,\dots,G_t)$ and let $c$ be any $t$-coloring of the $k$-chains of $C_{2^n}$. Fix a linear extension $\pi:B_n\to C_{2^n}$ and for a $k$-chain $A$ in $B_n$, let $c'(A)=c(\pi(A))$. By the definition of $n$, $c'$ must admit a copy of $G_i$ in color $i$ for some $i$; call this copy $H$. As $\pi$ is a linear extension, $\pi(H)$ is a copy of $G_i$ in color $i$ under $c$, so $\CR^1(G_1,\dots,G_t)\leq 2^n$.
\end{proof}

Let $G_1,\dots,G_t$ be $k$-uniform pographs.
If every linear extension of $G_i$ is isomorphic for all $i \in [t]$, then observe that $\CR^k(G_1,\dots,G_t) = \OR^k(G_1,\dots,G_t)$; pographs with this property include $\vee_r$, $\wedge_r$, $\diamond_r$, and $\bowtie_r^s$.
When every $G_i$ is totally-ordered, we have another equivalence of partially-ordered Ramsey numbers.

\begin{prop}\label{prop:totallyordered}
If $G_1,\dots,G_t$ are totally-ordered $k$-uniform pographs, then 
\[\BR^k(G_1,\dots,G_t) = \CR^k(G_1,\dots,G_t)- 1 = \OR^k(G_1,\dots,G_t)-1.\]
\end{prop}
\begin{proof}
The inequality $\BR^k(G_1,\dots,G_t) \leq \CR^k(G_1,\dots,G_t)-1$ follows from Proposition~\ref{prop:chainboundbool}.

Let $N = \CR^k(G_1,\dots,G_t)-1$ and let $c$ be a $t$-coloring of the $k$-chains in $C_N$ that does not contain a copy of $G_i$ in color $i$ for all $i \in [t]$.
Define a map $\rho : B_{N-1} \to C_N$ by $\rho(A) = |A| + 1$; if $A_1 \subset A_2 \subset \cdots\subset A_\ell$ is an $\ell$-chain in $B_{N-1}$, then $(\rho(A_1),\dots,\rho(A_\ell))$ is an $\ell$-chain in $C_N$.
For a $k$-chain $A_1 \subset A_2 \subset \cdots \subset A_k$ in the Boolean lattice $B_{N-1}$, let $c'(A_1,\dots,A_k) = c(\rho(A_1),\rho(A_2),\dots,\rho(A_k))$.
Consider a copy of $G_i$ in $B_{N-1}$. 
Since $G_i$ is totally-ordered, the elements of $G_i$ form a chain in $B_{N-1}$ and thus $\rho$ maps the elements of $G_i$ onto a copy of $G_i$ in $C_{N}$.
Since $c$ avoids $i$-colored copies of $G_i$ in $C_N$, so does $c'$ avoid $i$-colored copies of $G_i$ in $B_{N-1}$.
\end{proof}

The above argument requires that the vertices of a totally-ordered graph occupy distinct levels in any embedding of $G$ into the Boolean lattice.
If $G$ is not totally-ordered, then there is a pair of vertices which are incomparable; these two vertices may occupy the same level in an embedding of $G$ into $B_n$. 
It seems reasonable to expect that if $G$ contains large antichains, then the lower bound in Proposition \ref{prop:chainboundbool} should be closer to the truth.
We find this to be true for a few classes of pographs with large antichains.

\subsection{Matchings}
A natural class of $k$-uniform pographs with large antichains are those were the $k$-chains are completely independent, i.e.\ $k$-uniform matchings. In this case, we find the logarithmic bound on the Boolean Ramsey number is essentially tight and is off by at most $1$ in the 2-uniform case.

\begin{theorem}\label{thm:boolmatch}
Let $m_1\geq\dots\geq m_t$ and let $M_{m_1}^k,\dots,M_{m_t}^k$ be $k$-uniform matchings of size $m_1,\dots,m_t$.
\[
\lg\left(km_1+\sum_{i=2}^t(m_i-1)\right)\leq\BR^k(M_{m_1}^k,\dots,M_{m_t}^k)\leq\left\lceil\lg\left(1+\sum_{i=1}^t(m_i-1)\right)\right\rceil+k-1.
\]
\end{theorem}

\begin{proof}
\textit{Lower bound.}
Observe $\CR^k(M_{m_1}^k,\dots,M_{m_t}^k)=\R^k(M_{m_1}^k,\dots,M_{m_t}^k)$ as every copy of an unordered matching can be considered a linear extension of a partially-ordered matching as there are no additional relations between the elements in different edges.
Alon,  Frankl, and Lov\'asz \cite{AFL86} demonstrated that if $m_1\geq\dots\geq m_t$, then $\R^k(M_{m_1}^k,\dots,M_{m_t}^k)=km_1+\sum_{i=1}^t(m_i-1)$.
Apply Proposition \ref{prop:chainboundbool} to complete the lower bound.

\def\ext{\operatorname{ext}}
\textit{Upper bound.}
Let $N=\left\lceil\lg\left(1+\sum_{i=1}^t(m_i-1)\right)\right\rceil$ and let $c$ be a $t$-coloring of the $k$-chains in $B_{N+k-1}$.
Let $X$ be the family of subsets of $[N]$ within $B_{N+k-1}$.
For every set $A \in X$, define the \emph{extension} of $A$ to be the $k$-chain $\ext(A) = (A, A\cup \{N+1\}, A \cup \{N+2\},\dots, A \cup \{N+1,\dots,N+k-1\})$.
For $i \in [t]$, define the set $T_i$ to be the sets $A \in X$ where $c(\ext(A)) = i$.
Since $|X| = 2^N \geq 1 + \sum_{i=1}^k(m_i-1)$, the pigeonhole principle implies that $|T_i| \geq m_i$ for some $i$.
The collection of $k$-chains $\ext(A)$ for $A \in T_i$ form an $i$-colored matching of size at least $m_i$.
\end{proof}

Matchings are usually much simpler than other graphs.
Indeed, we limit our focus to 2-uniform pographs for the remainder of this section.

\subsection{Cups and Caps}

We now focus on  the Boolean Ramsey numbers of $r$-caps and $r$-cups.
To begin, the following proposition follows directly from the pigeonhole principle by considering all $r_i$-cups with minimum element $\varnothing$ or all $r_i$-caps with maximum element $[N]$.

\begin{prop}\label{prop:samecup}
For positive integers $r_1,\dots,r_t$,
\[
\BR^2(\vee_{r_1},\dots,\vee_{r_t}) =\BR^2(\wedge_{r_1},\dots,\wedge_{r_t}) = \left\lceil \lg\left(2 + \sum_{i=1}^t (r_i-1)\right)\right\rceil.
\]
\end{prop}

While the Boolean Ramsey number was simple to compute when considering a collection of cups or a collection of caps, the Ramsey numbers become more complicated when considering a collection of both cups and caps.
This next proposition states that knowing the $2$-color partially-ordered Ramsey number for cup verses cap is sufficient to determine the multicolor Ramsey number.

\begin{prop}\label{prop:multicupcap}
Let $R=1+\sum_{i=1}^n(r_i-1)$ and $S=1+\sum_{i=1}^m(s_i-1)$.
\begin{align*}
\CR^2(\vee_{r_1},\dots,\vee_{r_n},\wedge_{s_1},\dots,\wedge_{s_m}) &=\CR^2(\vee_R,\wedge_S),\quad \text{and}\\
\BR^2(\vee_{r_1},\dots,\vee_{r_n},\wedge_{s_1},\dots,\wedge_{s_m}) &=\BR^2(\vee_R,\wedge_S).
\end{align*}
\end{prop}

\begin{proof}
We prove equality by demonstrating both inequalities. 

($\leq$)
Consider an $(n+m)$-coloring $c$ of the edges of either a chain $C_N$ or a Boolean lattice $B_N$.
Let $c'$ be a 2-coloring where $c'(e) = 1$ if $c(e) \leq n$ and $c'(e) = 2$ if $c(e) > n$.
If $c$ avoids $i$-colored copies of $\vee_{r_i}$ and $(n+j)$-colored copies of $\wedge_{s_j}$, then $c'$ avoids 1-colored copies of $\vee_R$ and 2-colored copies of $\wedge_S$.

($\geq$)
Let $c$ be a 2-coloring  of the edges of either a chain $C_N$ or a Boolean lattice $B_N$ and suppose that $c$ does not contain a copy of $\vee_R$ in color 1 or a copy of $\wedge_S$ in color 2.
We will construct an $(n+m)$-coloring $c$.
Since $c$ does not contain a 1-colored copy of $\vee_R$, we have $|\up_1(v)| < R$; partition $\up_1(v)$ into $n$ parts $P_1 \cup \dots \cup P_n$ such that $|P_i| \leq r_i-1$ and let $c'(vu) = i$ if $u \in P_i$.
Since $c$ does not contain a 2-colored copy of $\wedge_S$, we have $|\down_2(v)| < S$; partition $\down_2(v)$ into $m$ parts $P_1 \cup \dots \cup P_m$ such that $|P_j| \leq s_j-1$ and let $c'(vu) = n+j$ if $u \in P_j$.
Every edge is colored exactly once by the process above and hence $c'$ avoids $i$-colored copies of $\vee_{r_i}$ and $(n+j)$-colored copies of $\wedge_{s_j}$.
\end{proof}

Choudum and Ponnusamy~\cite{CP02} determined $\CR^2(\vee_r,\wedge_s)$ exactly.

\begin{theorem}[Choudum and Ponnusamy~\cite{CP02}]\label{thm:cpcapcup}
For integers $r,s\geq 2$,
\[
\CR^2(\vee_r,\wedge_s)=\left\lfloor {\sqrt{1+8(r-1)(s-1)}-1\over 2}\right\rfloor+r+s.
\]
\end{theorem}

Observe that this implies $\CR^2(\vee_r,\wedge_s) \leq (1+\sqrt{2})(r+s)$.
Therefore, by applying Proposition \ref{prop:multicupcap}, we see that 
\[
\CR^2(\vee_{r_1},\dots,\vee_{r_n},\wedge_{s_1},\dots,\wedge_{s_m})=\left\lfloor{\sqrt{1+8(R-1)(S-1)}-1\over 2}\right\rfloor +R+S \leq (1+\sqrt{2})(R+S),
\]
where $R=1+\sum_{i=1}^n(r_i-1)$ and $S=1+\sum_{i=1}^m(s_i-1)$.

In contrast to the linear bound of chain Ramsey numbers, the following theorem shows that the Boolean Ramsey numbers for cups and caps is logarithmic.

\begin{theorem}\label{thm:boolcupcap}
For integers $r,s\geq 2$, 
\[
\lg\left(\left\lfloor {\sqrt{1+8(r-1)(s-1)}-1\over 2}\right\rfloor+r+s\right)\leq\BR^2(\vee_r,\wedge_s)\leq \left\lceil\log_{3/2}(r+s-1)\right\rceil.
\]
\end{theorem}

\begin{proof}
\textit{Lower Bound.} 
The lower bound follows from Theorem \ref{thm:cpcapcup} and applying Proposition \ref{prop:chainboundbool}.

\textit{Upper Bound.} 
Let $N=\left\lceil\log_{3/2}(r+s-1)\right\rceil$ and suppose that $c$ is a $2$-coloring of the edges of $B_N$ that avoids copies of $\vee_r$ in color $1$ and avoids copies of $\wedge_s$ in color $2$.
Thus, for any $v\in B_N$, $|\up_1(v)|\leq r-1$ and $|\down_2(v)|\leq s-1$.
In particular, this implies that $|\down_1(v)|=|\down(v)|-1-|\down_2(v)|\geq 2^{|v|}-s$. 

Let $W=B_N\setminus\{[N]\}$ and let $T$ be the set of elements $v$ in $W$ where $|\up_1(v)\cap W| = r-1$.
As $c$ avoids copies of $\vee_r$ in color $1$, for any $v\in T$, $c(v,[N])=2$.
Hence, $|T|\leq s-1$ since $c$ avoids copies of $\wedge_s$ in color $2$.

Let $b$ be the number of edges $uv$ with $c(uv) = 1$ and both $u$ and $v$ are in $W$, then
\[
b = \sum_{v\in W}|\down_1(v)|
 \geq \sum_{v\in W}(2^{|v|}-s)
 = \sum_{i=0}^{n-1}{N\choose i}2^i-s(2^N-1)
 = 3^N-2^N(s+1)+s.
 \]

On the other hand,
\begin{align*}
b &= \sum_{v\in W}|\up_1(v)\cap W|\\
 &=\sum_{v\in T}(r-1)+\sum_{v\in W\setminus T}|\up_1(v)\cap W|\\
 &\leq  |T|(r-1)+(2^N-1-|T|)(r-2)\\
 &= |T|+(2^N-1)(r-2)\\
 &\leq s-1+(2^N-1)(r-2).
\end{align*}

Therefore, $3^N-2^N(s+1)+s\leq b\leq s-1+(2^N-1)(r-2)$, so
\[
\left({3\over 2}\right)^N\leq r+s-1-(r-1)2^{-N}<r+s-1.
\]
This, however, is a contradiction as $N=\left\lceil\log_{3/2}(r+s-1)\right\rceil$.
\end{proof}

By applying Proposition \ref{prop:multicupcap}, observe that 
\[
\lg\left(\left\lfloor {\sqrt{1+8(R-1)(S-1)}-1\over 2}\right\rfloor+R+S\right) \leq\BR^2(\vee_{r_1},\dots,\vee_{r_n},\wedge_{s_1},\dots,\wedge_{s_m})\leq \left\lceil\log_{3/2}(R+S-1)\right\rceil, 
\]
where $R=1+\sum_{i=1}^n(r_i-1)$ and $S=1+\sum_{i=1}^m(s_i-1)$.

\subsection{Diamonds}
An $r$-diamond combines the behavior of an $r$-cup with an $r$-cap.
Despite doubling the number of edges in the pograph, we find similar logarithmic behavior in the Boolean Ramsey numbers.
However, our methods focus on the 2-color case and fail to extend to the generic $t$-color case.

Using Theorem \ref{thm:cpcapcup}, Balko, Cibulka, Kr\'{a}l and Kyn\u{c}l~\cite{BCKK13} argued that $11\leq\CR_2^2(\diamond_2)\leq 13$ and show that the lower bound is tight with computer assistance. The following is a direct extension of their argument.

\begin{theorem}\label{thm:chaindiamond}
If $r,s\geq 2$, then
\[
\CR^2(\diamond_r,\diamond_s)\leq 2\cdot\left\lfloor {\sqrt{1+8(r-1)(s-1)}-1\over 2}\right\rfloor+3(r+s)-1
\]
\end{theorem}

\begin{proof}
Let $N=2\cdot\left\lfloor {\sqrt{1+8(r-1)(s-1)}-1\over 2}\right\rfloor+3(r+s)-1$ and suppose that $c$ is a $2$-coloring of the edges of $C_N$ that avoids copies of $\diamond_r$ in color $1$ and avoids copies of $\diamond_s$ in color $2$.
Therefore, $|\up_1(1)\cap\down_1(N)|\leq r-1$ and $|\up_2(1)\cap\down_2(N)|\leq s-1$.
Hence, $|\up_1(1)\cap\down_2(N)|+|\up_2(1)\cap\down_1(N)|\geq (N-2)-(r-1)-(s-1)=N-r-s$.
By the pigeonhole principle, there is some $i\in \{1,2\}$ for which 
\begin{align*}
|\up_i(1)\cap\down_{3-i}(N)| &\geq\left\lceil{N-r-s\over 2}\right\rceil\\
&=\left\lceil\left\lfloor {\sqrt{1+8(r-1)(s-1)}-1\over 2}\right\rfloor+r+s-{1\over 2}\right\rceil\\
&=\left\lfloor {\sqrt{1+8(r-1)(s-1)}-1\over 2}\right\rfloor+r+s \\
&=\CR^2(\wedge_r,\vee_s).
\end{align*}
If this is true for $i=1$, then $c$ restricted to $\up_1(1)\cap\down_2(N)$ must admit either a $\vee_s$ in color $2$, in which case $c$ admits a $\diamond_s$ in color $2$, or a $\wedge_r$ in color $1$, in which case $c$ admits a $\diamond_r$ in color $1$.
A similar contradiction is found if the inequality holds for $i=2$.
\end{proof}

\begin{cor}
If $s, r\geq 2$, then $\OR^2(\diamond_s,\diamond_r) = \CR^2(\diamond_s,\diamond_r) \leq (3 + \sqrt{2})(r+s) \approx 4.414(r+s).$
\end{cor}

This upper bound is asymptotically correct, up to the leading constant.

\begin{prop}\label{prop:chaindiamondlower}
If $s \geq r \geq 2$, then $\OR^2(\diamond_s,\diamond_r) = \CR^2(\diamond_s,\diamond_r) > 2s+2$.
\end{prop}

\begin{proof}
Let $N = 2s+2$ and consider $X_1 = \{1,\dots,s+1\}$ and $X_2 = \{s+2,\dots,N\}$.
If an edge has both endpoints in $X_i$ for some $i$, then color that edge with color 1. 
If an edge has one endpoint in $X_1$ and another in $X_2$, then color that edge with color 2.
Observe that there is no $\diamond_r$ in color 2, as there is no chain of length 2 in color 2.
Further, there is no $\diamond_s$ in color 1, as such a subgraph would be entirely contained in $X_1$ or $X_2$, but these sets have size $s+1$ and $|V(\diamond_s)| = s+2$.
\end{proof}

Note that Proposition \ref{prop:chaindiamondlower} immediately implies that if $s\geq r\geq 2$, then $\BR^2(\diamond_s,\diamond_r)\geq\lg(2s+3)$.

To investigate and upper bound on the Boolean Ramsey numbers of diamonds, we first consider diamonds and cups (Theorem~\ref{thm:booldiamondcup}) before completing the argument for two diamonds (Theorem~\ref{thm:booldiamondupper}).

\begin{theorem}\label{thm:booldiamondcup}
Let $s,r \geq 2$ be integers.
\[
	\BR^2(\diamond_s,\vee_r) \leq \BR^2(\wedge_{s+r},\vee_r) \leq \left\lceil\log_{3/2}(2r+s-1) \right\rceil
\]
\end{theorem}

\begin{proof}
The second inequality holds by Theorem~\ref{thm:boolcupcap}.
Let $N = \BR^2(\wedge_{s+r},\vee_r)$ and consider a 2-coloring of the edges of $B_N$ and suppose the 2-coloring does not contain an $r$-cup in color 2.
Therefore, there is an $(s+r)$-cap in color 1.
Let $A_0,A_1,\dots,A_{s+r-1}, B$ be the sets in this cap where $A_i \subseteq B$ for all $i$.
If the empty set is in the cap, then let $A_0 = \varnothing$.
There are $s+r-1$ edges from the empty set to the sets $A_i$ with $i \in \{1,\dots,s+r-1\}$.
Since the coloring avoids $r$-cups in color 2, there must be at least $s$ sets $A_i$ such that the edge $(\varnothing,A_i)$ has color 1.
Thus, these $A_i$'s along with the empty set and $B$ forms an $s$-diamond of color 1.
\end{proof}

\begin{theorem}\label{thm:booldiamondupper}
Let $s, r \geq 2$ be integers.
\[
	\BR^2(\diamond_s,\diamond_r) \leq \BR^2(\diamond_s, \vee_{s+r-1}) + \left\lceil\lg(2s+2r)\right\rceil \leq 2\left\lceil\log_{3/2}(2r+2s-1)\right\rceil.
\]
\end{theorem}

\begin{proof}
The second inequality holds by Theorem~\ref{thm:booldiamondcup} and logarithmic identities.
Let $N = \BR^2(\diamond_s,\vee_{s+r-1})$ and $M = \left\lceil\lg(2s+2r)\right\rceil$.
Consider a 2-coloring $c$ of the edges of $B_{N + M}$.
Suppose for the sake of contradiction that $c$ does not contain an $s$-diamond in color 1 and does not contain an $r$-diamond in color 2.

For $j \in \{1,2\}$, let $I_j$ contain the sets $Z$ such that $[N] \subset Z \subset [N+M]$. and $c(Z,[N+M]) = j$.
Since $|I_1\cup I_2| = 2^M-2 \geq 2s+2r-2$, either $|I_1| \geq s+r-1$ or $|I_2| \geq s+r-1$.
We will assume that $|I_1| \geq s+r-1$; the other case follows by a symmetric argument.
Let $I \subseteq I_1$ with $|I| = s+r-1$.

Since $N = \BR^2(\diamond_s,\vee_{s+r-1})$, and $c$ does not contain an $s$-diamond in color 1, there exists an $(s+r-1)$-cup of color $2$ in $\down([N])$.
Let $A_0, B_1,\dots,B_{s+r-1}$ be the sets of this cup such that $A_0 \subset B_j$ for each $j \in \{1,\dots, s+r-1\}$.
Notice that $B_j \subset Z$ for all $j\in \{1,\dots,s+r-1\}$ and all $Z \in I$.
The edges between $B_1,\dots,B_{s+r-1}$ and the sets $Z \in I$ form a 2-colored copy of the complete bipartite graph $K_{s+r-1,s+r-1}$.

For every $j \in \{1,\dots,s+r-1\}$ there are at most $s-1$ edges of color $1$ from $B_j$ to the sets $Z \in I$, since $c$ avoids $s$-diamonds in the color $1$.
For every $Z \in I$, there are at most $r-1$ edges of color $2$ from the sets $B_1,\dots,B_{s+r-1}$ to $Z$, since $c$ avoids $r$-diamonds in the color $2$.
However, this implies that the total number of edges in this complete bipartite graph is at most $(s+r-1)((s-1)+(r-1)) < (s+r-1)^2$, a contradiction.
\end{proof}

Using Theorems~\ref{thm:booldiamondcup} and \ref{thm:booldiamondupper}, we find $\BR^2(\diamond_2,\vee_3) \leq 5$ and $\BR^2(\diamond_2,\diamond_2) \leq 8$.
With a more specialized argument for the case $r = s= 2$, one can prove $\BR^2(\diamond_2,\diamond_2) \leq \BR^2(\diamond_2,\vee_3) + 2$, but this is not tight.
In the next section, we discuss computational methods to compute Boolean Ramsey numbers, and we verify that $\BR^2(\diamond_2,\vee_3) = 4$ and $\BR^2(\diamond_2,\diamond_2)=5$.

\section{Computational Results}\label{sec:computation}

Ramsey numbers are difficult to compute in all but the simplest of cases.
A na\"ive algorithm for testing $\R_t^k(G) > n$ takes $O(t^{n^k})$ steps, and advanced algorithm techniques do not improve on the asymptotic growth of this method.
However, using the same method to test $\BR_t^k(G) > n$ can require $O(t^{(k+1)^n})$ steps.
In fact, simply storing a $t$-coloring of the $k$-chains in $B_n$ requires $(k+1)^n\lg t$ bits of space.
This makes finding exact values of 2-color, 2-uniform Boolean Ramsey numbers very difficult once $n \geq 5$.

To test if $\BR^2(H_1,H_2) > n$, we use a SAT formulation to determine if there exists a 2-coloring $c$ of the comparable pairs in $B_n$ that avoids copies of $H_1$ in color 1 and avoids copies of $H_2$ in color 2.
For every comparable pair $A \subset B$, we let $x_{A,B}$ be a Boolean variable; the variable $x_{A,B}$ is true exactly when $c(A,B) = 1$.
For every copy of $H_1$ in $B_n$, we create a constraint that requires at least one variable $x_{A,B}$ to be false among the edges $(A,B)$ in the copy of $H_1$.
Similarly, for every copy of $H_2$ in $B_n$, we create a constraint that requires at least one variable $x_{A,B}$ to be true among the edges $(A,B)$ in the copy of $H_2$.
There exists such a 2-coloring if and only if these constraints can be simultaneously satisfied.

We used a similar SAT formulation to demonstrate that $\BR^1(B_n,B_m) > n + m -1$ (formulation is satisfiable) and $\BR^1(B_n,B_m) \leq n +m$ (formulation is unsatisfiable) when $3 \geq n \geq m \geq 1$ and $n$ and $m$ are not both $3$.

We used Sage~\cite{Sage} to construct our SAT formulations in SMT2 format.
We then used the Microsoft Z3~\cite{Z3} SMT\footnote{Satisfiability Modulo Theory.} solver to test the formulations.
The results are summarized in Table~\ref{tab:computation}.
These computations were completed using a standard laptop computer with each test taking at most a few hours.
All Sage code and SAT formulations are available online\footnote{See \url{http://orion.math.iastate.edu/dstolee/data.htm} for all code and data.}.

This method was limited by the exponential growth in the size of the formulations more than the time it takes to solve them.
We selected only a few examples to test with $n = 5$ due to the number of copies of the pographs $H_1$ and $H_2$ that appeared within $B_5$.
We could test $\BR^2(B_2,\diamond_2) = \BR^2(B_2,B_2) = 6$ due to the fact that $B_2$  and $\diamond_2$ have only four elements, which greatly limited the number of copies appear within $B_6$, but these tests were our largest computations.

\begin{table}[tp]
\centering
\subfigure[$\BR^2(\vee_r,\wedge_s)$.]{\begin{tabular}[h]{r||c|c|c|c}& $\wedge_2$ & $\wedge_3$ & $\wedge_4$ & $\wedge_5$ \\\hline&&&&\\[-2.5ex]\hline
$\vee_2$  & 3 & 3 & 4 & 4\\\hline
$\vee_3$  & 3 & 4 & 5 & 5\\\hline
$\vee_4$  & 4 & 5 & 5 & \\\hline
$\vee_5$  & 4 & 5 &  & \\\end{tabular}}
\qquad
\subfigure[$\BR^2(H_1,H_2)$.]{\begin{tabular}[h]{r||c|c}& $\diamond_2$& $B_2$\\\hline&&\\[-2.5ex]\hline
$\diamond_2$  & 5 & 6\\\hline
$B_2$  & 6 & 6\\\end{tabular}}
\qquad
\subfigure[~$\BR^2(C_r,\wedge_s)=~\BR^2(C_r,\vee_s)$]{\begin{tabular}[h]{r||c|c|c}& $\wedge_2$ & $\wedge_3$ & $\wedge_4$\\\hline&&&\\[-2.5ex]\hline
$C_2$  & 2 & 2 & 3\\\hline
$C_3$  & 3 & 3 & 4\\\hline
$C_4$  & 4 & 4 & 5\\\end{tabular}}\\

\subfigure[$\BR^2(M_r,M_s)$.]{\begin{tabular}[h]{r||c|c|c}& $M_2$& $M_3$& $M_4$\\\hline&&&\\[-2.5ex]\hline
$M_2$  & 3 & 3 & 4\\\hline
$M_3$ &  & 3 & 4\\\hline
$M_4$  &  &  & 4\\\end{tabular}}
\qquad
\subfigure[$\BR^2(C_r,M_s)$]{\begin{tabular}[h]{r||c|c|c}& $M_2$ & $M_3$& $M_4$\\\hline&&&\\[-2.5ex]\hline
$C_2$  & 2 & 3 & 3\\\hline
$C_3$  & 3 & 4 & 4\\\hline
$C_4$  & 4 & 5 & \\\hline
$C_5$  & 5 &  & \\\end{tabular}}
\qquad
\subfigure[~$\BR^2(\vee_r,M_s)=~\BR^2(\wedge_r,M_s)$.]{\begin{tabular}[h]{r||c|c|c}& $M_2$ & $M_3$ & $M_4$ \\\hline&&&\\[-2.5ex]\hline
$\vee_2$  & 3 & 3 & 4\\\hline
$\vee_3$  & 3 & 4 & 4\\\hline
$\vee_4$  & 4 & 4 & 4\\\hline
$\vee_5$  & 4 & 4 & 4\\\hline
$\vee_6$  & 4 & 4 & 4\\\end{tabular}}\\

\subfigure[$\BR^2(\vee_r,\bowtie_m^n)=\BR^2(\wedge_r,\bowtie_n^m)$]{\hspace{0.25in}\begin{tabular}[h]{r||c|c|c|c}& $\bowtie_2^2$ & $\bowtie_2^3$& $\bowtie_3^2$ & $\bowtie_3^3$ \\\hline&&&&\\[-2.5ex]\hline
$\vee_2$  & 4 & 4 & 4 & 5\\\hline
$\vee_3$  & 4 & 4 & 5 & 5\\\hline
$\vee_4$  & 4 & 5 &  & \\\end{tabular}\hspace{0.25in}}
\qquad
\subfigure[~$\BR^2(\vee_r,W_s)=~\BR^2(\wedge_r,W_s)$]{\begin{tabular}[h]{r||c|c|c|c}& $W_2$& $W_3$& $W_4$& $W_5$\\\hline&&&&\\[-2.5ex]\hline
$\vee_2$  & 4 & 4 & 4 & 4\\\hline
$\vee_3$  & 4 & 4 & 4 & 4\\\hline
$\vee_4$  & 4 & 5 &  & \\\end{tabular}}
\qquad
\subfigure[~$\BR^2(H,\wedge_s)=~\BR^2(H,\vee_s)$.]{\begin{tabular}[h]{r||c|c|c}& $\wedge_2$ & $\wedge_3$ & $\wedge_4$ \\\hline&&&\\[-2.5ex]\hline
$\diamond_2$  & 4 & 4 & 5\\\hline
$B_2$  & 4 & 4 & 5\\\hline
$\diamond_3$  & 4 & 5 & 5\\\end{tabular}}

\caption{\label{tab:computation}Computational results for small 2-uniform Boolean Ramsey numbers.}
\end{table}

A highly specialized algorithm may be able to extend these results to more examples when $n = 6$, but we expect this will be very difficult.

\section{Other Poset Families}\label{sec:relations}

While we have mainly focused on chain Ramsey numbers and Boolean Ramsey numbers, many other families of posets can give rise to interesting Ramsey numbers. 

\subsection{Generic Poset Families}
Let $\P = \{ P_n : n \geq 1\}$ be a poset family with $P_n\subseteq P_{n+1}$ for all $n$.
For a $t$-tuple $(G_1,\dots,G_t)$ of $k$-uniform pographs, we say that $\P$ is \emph{$k$-Ramsey for $(G_1,\dots,G_t)$} if there exists an $N$ such that every $t$-coloring of the $k$-chains in $P_N$ contains an $i$-colored copy of $G_i$ for some $i$.
The partially-ordered Ramsey number $\PR{\P}^k(G_1,\dots,G_t)$ exists exactly when $\P$ is $k$-Ramsey for $(G_1,\dots,G_t)$.

We say a family $\P$ is a \emph{universal poset family} if $\P$ is $k$-Ramsey for every $t$-tuple of $k$-uniform pographs and every $k \geq 1$.
If the height of $P_n$ grows without bound, then $\P$ is a universal poset family as eventually $P_n$ contains a chain of order $\CR^k(G_1,\dots,G_t)$ for any $G_1,\dots, G_t$. In fact, universal poset families are exactly those where $h(P_n)$ is unbounded.
Some of our results hold for universal poset families, such as Propositions~\ref{prop:multicupcap}.
Other results must be generalized slightly,  such as the following generalization of Proposition~\ref{prop:chainboundbool}.

\begin{prop}\label{prop:chainbound}
Let $\P = \{ P_n : n \geq 1\}$ be a universal poset family.
Define $s_{\P}(n)$ to be the minimum $N$ such that $|P_N| \geq n$.
Define $h_{\P}(n)$ to be the minimum $N$ such that $h(P_N)\geq n$.
Then,
\[
	s_{\P}( \CR^k(G_1,\dots,G_t) ) \leq \PR{\P}^k(G_1,\dots,G_t) \leq h_{\P}(\CR^k(G_1,\dots,G_t)).
\]
\end{prop}

Using the function $h_{\P}(n)$, one can restate Proposition~\ref{prop:totallyordered} as $\PR{\P}^k(G_1,\dots,G_t) = h_{\P}(\CR^k(G_1,\dots,G_t))$ for totally-ordered graphs $G_1,\dots,G_t$.

\subsection{Rooted Bipartite Ramsey Numbers}
A poset family does not need to be universal in order to be interesting.
Consider the family $\K = \{ \bowtie_n^n : n \geq 1\}$ of $n,n$-butterfly posets.
This family is not universal since $C_3 \not\subseteq\, \bowtie_n^n$ for any $n$.
However, we can still consider $G_1,\dots,G_t$ to be pographs whose posets partition into two antichains $V(G_i) = X_i\cup Y_i$ where every $x \in X_i$ is comparable to at least one element $y \in Y_i$ with $x \leq y$.
In this case, the Ramsey number $\PR{\K}^2(G_1,\dots,G_t)$ is the minimum $N$ such that every $t$-coloring of the edges of the complete bipartite graph $K_{N,N}$ with vertex set $V(K_{N,N}) = A\cup B$ contains an $i$-colored copy of the bipartite graph $G_i$ where $X_i \subseteq A$ and $Y_i \subseteq B$ for some $i$. 

If we remove the condition that $X_i \subseteq A$ and $Y_i \subseteq B$, then this Ramsey problem is identical to finding \emph{bipartite Ramsey numbers} (see~\cite{beineke1975bipartite,conlon2008new,goddard2000bipartite,hattingh1998bipartite,irving1978bipartite}).
The equivalent of the Tur\'an problem in this context is called the \emph{Zarenkiewicz problem} (see~\cite{furedi1996upper,goddard2000bipartite,irving1978bipartite}).
The most widely studied version of these numbers are those where $G_i = \bowtie_n^m$ for some $n, m$.

With the condition that $X_i \subseteq A$ and $Y_i \subseteq B$, we can call $\PR{\K}^2(G_1,\dots,G_t)$ the \emph{rooted} bipartite Ramsey number.
In this case, it may be true that $\PR{\K}^2(\bowtie_r^s,\bowtie_r^s) \neq \PR{\K}^2(\bowtie_r^s,\bowtie_s^r)$ when $r \neq s$.
The final paragraph of the proof of Thereom~\ref{thm:booldiamondupper} implicitly proves and uses the fact that $\PR{\K}^2(\wedge_s,\vee_r) = \PR{\K}^2(\vee_s,\vee_r) = s+r-1$.

\subsection{High-Dimensional Grids}

Closely related to the Boolean lattice is the \emph{$m$-dimensional $\ell$-grid} $([\ell]^m,\preceq)$, whose elements are $m$-tuples $(x_1,\dots,x_n)$ where every coordinate $x_i$ is in the set $[\ell]$, and $(x_1,\dots,x_n)\preceq(y_1,\dots,y_n)$ if and only if $x_i\leq y_i$ for all $i$ (in particular, the Boolean lattice $B_n$ corresponds to $[2]^n$).
When constructing a universal poset family $\P = \{P_n : n \geq 1\}$ from these grids, we have two natural options for the parameter $n$.
First, we could have the dimension grow with $n$: let $Q_n(\ell) = [\ell]^n$ and $\Q(\ell) = \{ Q_n(\ell) : n \geq 1\}$.
Second, we could have the length grow with $n$: let $H_n(m) = [n]^m$ and $\cH(m) = \{ H_n(m) : n \geq 1\}$.
Along these lines, we provide analogues of theorems from Section \ref{sec:2unif} for each of these cases.

\begin{theorem}[Analogue of Theorem \ref{thm:boolcupcap}]
For $s,r\geq 2$,
\begin{align*}
\log_\ell\left(\left\lfloor {\sqrt{1+8(r-1)(s-1)}-1\over 2}\right\rfloor+r+s\right)\leq\PR{\Q(\ell)}^2(\vee_r,\wedge_s)\leq\left\lceil\log_{(\ell+1)/2}(r+s-1)\right\rceil & \quad\text{and}\\
\left(\left\lfloor {\sqrt{1+8(r-1)(s-1)}-1\over 2}\right\rfloor+r+s\right)^{1/m}\leq \PR{\cH(m)}^2(\vee_r,\wedge_s)\leq  \left\lceil 2(r+s-1)^{1/m}\right\rceil-1.
\end{align*}
\end{theorem}

\begin{theorem}[Analogue of Theorem \ref{thm:booldiamondcup}]
For $s,r\geq 2$,
\begin{align*}
\PR{\Q(\ell)}^2(\diamond_s,\vee_r)\leq\PR{\Q(\ell)}^2(\wedge_{s+r},\vee_r)\leq\left\lceil\log_{(\ell+1)/2}(2r+s-1)\right\rceil & \quad\text{and}\\
\PR{\cH(m)}^2(\diamond_s,\vee_r)\leq\PR{\cH(m)}^2(\wedge_{s+r},\vee_r)\leq  \left\lceil 2(2r+s-1)^{1/m}\right\rceil-1.
\end{align*}
\end{theorem}

\begin{theorem}[Analogue of Theorem \ref{thm:booldiamondupper}]
For $s,r\geq 2$,
\begin{align*}
\PR{\Q(\ell)}^2(\diamond_s,\diamond_r)\leq\PR{\Q(\ell)}^2(\diamond_r,\vee_{s+r-1})+\lceil\log_{\ell}(2s+2r)\rceil\leq 2\left\lceil\log_{(\ell+1)/2}(2r+2s-1)\right\rceil & \quad\text{and}\\
\PR{\cH(m)}^2(\diamond_s,\diamond_r)\leq\PR{\cH(m)}^2(\diamond_s,\vee_{s+r-1})+\left\lceil(2s+2r)^{1/m}\right\rceil\leq  3\left\lceil(2r+2s-1)^{1/m}\right\rceil.
\end{align*}
\end{theorem}

The proof of each of these theorems are identical to their analogues in the Boolean lattice. Notice that in each case, the Ramsey number is within a constant factor of the lower bound given in Proposition \ref{prop:chainbound}. It would be of interest to explore other partially-ordered Ramsey numbers using $\Q(\ell)$ or $\cH(m)$ as the host family.

\section{Future Work}

For 1-uniform Boolean Ramsey numbers, the main open question is to determine $\BR^1_t(B_d)$. We showed that $td\leq\BR^1_t(B_d)\leq 2td^2$ and that $\BR^1_2(B_d)\geq 2d+1$ for $3\leq d\leq 8$ and $d\geq 13$. It is reasonable to expect that if $|P|$ is large compared to $\levels(P)$, then $\BR^1_t(P)$ is closer to the lower bound of $t\cdot\levels(P)$ given in Proposition~\ref{prop:1ulow}. To this end, we pose the following two questions.

\begin{question}
Is $\BR^1_t(B_d)$ linear in $d$?
\end{question}

\begin{question}
For a poset $P$, is there a constant $c=c(P)$ such that $\BR^1_t(P)=t\cdot\levels(P)+c$?
\end{question}

It is important to point out that Theorem \ref{thm:mlubell} employed only the bound $\lu_n^{(m)}(\F)\leq\max_{\C_m}|\F\cap\C_m|$. 
It would be interesting to explore the actual value of $L_n^{(m)}(P)$ for specific posets $P$ as, in addition to implying bounds on 1-uniform Boolean Ramsey numbers, it may also lead to improvements on the bounds in the Tur\'an problem.

When it comes to higher uniformities, we are especially interested in the properties of partially-ordered graphs whose Boolean Ramsey numbers are within a constant factor of the lower bound given in Proposition \ref{prop:chainboundbool}.
In particular, we ask the following.

\begin{question}
What properties must a graph $G$ have so that the lower bound on the Boolean Ramsey number of $G$ given in Proposition \ref{prop:chainboundbool} is tight up to a constant?
\end{question}

We suspect that the answer to this question will focus on the properties of the underlying poset of $G$ and have very little to do with the actual edges of $G$.
In particular, we suspect that the answer relies heavily on the number and/or size of the antichains in the underlying poset.

Beyond this, an exploration of $\BR_t^2(B_d)$ is of great interest.
By applying the well-known bounds on $\R_2^2(K_d)$, we immediately observe that $\Omega(2^{d/2})\leq\BR_2^2(B_d)\leq O(4^{2^d})$.
We believe the upper bound to be far from the truth and would expect only an exponential bound, but any improvement to either bound would be of interest.

Finally, we pose the following question.

\begin{question}
For a $k$-uniform pograph $G$, what is the least integer $N$ such that any $t$-coloring of the $k$-chains of $B_N$ contains a monochromatic copy of $G$ such that the underlying poset of $G$ is induced?
\end{question}

For $1$-uniform pographs, this question has received attention; however, for $k\geq 2$, it is not even obvious that such an $N$ exists.

\section*{Acknowledgements}

The authors would like to thank Mikhail Lavrov for recommending SAT solvers as a method for computing small Boolean Ramsey numbers.
Thanks also to Maria Axenovich for discussing previous work on induced Boolean Ramsey numbers.

\bibliographystyle{abbrv}
\bibliography{references}

\begin{thebibliography}{10}

\bibitem{AFL86}
N.~Alon, P.~Frankl, and L.~Lov{\' a}sz.
\newblock The chromatic number of {K}neser hypergraphs.
\newblock {\em Transactions of the American Mathematical Society},
  298(1):359--370, Nov. 1986.

\bibitem{axenovich2015boolean}
M.~Axenovich and S.~Walzer.
\newblock Boolean lattices: Ramsey properties and embeddings., 2015.
\newblock Available as arXiv:1512.05565 [math.CO].

\bibitem{BCKK13}
M.~Balko, J.~Cibulka, K.~Kr\'{a}l, and J.~Kyn\v{c}l.
\newblock {R}amsey numbers of ordered graphs, Oct. 2013.
\newblock Available as arXiv:1310.7208 [math.CO].

\bibitem{beineke1975bipartite}
L.~W. Beineke and A.~J. Schwenk.
\newblock On a bipartite form of the {R}amsey problem.
\newblock In {\em Proceedings of the Fifth British Combinatorial Conference
  (Univ. Aberdeen, Aberdeen, 1975)}, pages 17--22, 1975.

\bibitem{CP02}
S.~Choudum and B.~Ponnusamy.
\newblock Ordered {R}amsey numbers.
\newblock {\em Discrete Mathematics}, 247(1-3):79--92, Mar. 2002.

\bibitem{CGKVV14}
J.~Cibulka, P.~Gao, M.~Kr\v{c}\'{a}l, T.~Valla, and P.~Valtr.
\newblock On the geometric {R}amsey number of outerplanar graphs.
\newblock {\em Discrete \& Computational Geometry}, 53(1):64--79, Nov. 2014.

\bibitem{conlon2008new}
D.~Conlon.
\newblock A new upper bound for the bipartite {R}amsey problem.
\newblock {\em Journal of Graph Theory}, 58(4):351--356, 2008.

\bibitem{CFLS14}
D.~Conlon, J.~Fox, C.~Lee, and B.~Sudakov.
\newblock Ordered {R}amsey numbers, Oct. 2014.
\newblock Available as arXiv:1410.5292 [math.CO].

\bibitem{CS15}
C.~Cox and D.~Stolee.
\newblock Ordered {R}amsey numbers of loose paths and matchings.
\newblock {\em Discrete Mathematics}, Nov. 2015.

\bibitem{DK07}
A.~De~Bonis and G.~O. Katona.
\newblock Largest families without an $r$-fork.
\newblock {\em Order}, 24(3):181--191, 2007.

\bibitem{DKS05}
A.~De~Bonis, G.~O. Katona, and K.~J. Swanepoel.
\newblock Largest family without ${A} \cup {B} \subseteq {C} \cap {D}$.
\newblock {\em Journal of Combinatorial Theory, Series A}, 111(2):331--336,
  2005.

\bibitem{Z3}
L.~De~Moura and N.~Bj{\o}rner.
\newblock Z3: An efficient smt solver.
\newblock In {\em Tools and Algorithms for the Construction and Analysis of
  Systems}, pages 337--340. Springer, 2008.

\bibitem{duffus1991fibres}
D.~Duffus, H.~A. Kierstead, and W.~T. Trotter.
\newblock Fibres and ordered set coloring.
\newblock {\em Journal of Combinatorial Theory, Series A}, 58(1):158--164,
  1991.

\bibitem{ES35}
P.~Erd{\H o}s and G.~Szekeres.
\newblock A combinatorial problem in geometry.
\newblock {\em Compositio Mathematica}, 2:463--470, 1935.

\bibitem{FPSS12}
J.~Fox, J.~Pach, B.~Sudakov, and A.~Suk.
\newblock {E}rd{\H o}s-{S}zekeres-type theorems for monotone paths and convex
  bodies.
\newblock {\em Proceedings of the London Mathematical Society},
  105(5):953--982, May 2012.

\bibitem{furedi1996upper}
Z.~F{\"u}redi.
\newblock An upper bound on {Z}arankiewicz'problem.
\newblock {\em Combinatorics, Probability and Computing}, 5(01):29--33, 1996.

\bibitem{goddard2000bipartite}
W.~Goddard, M.~A. Henning, and O.~R. Oellermann.
\newblock Bipartite {R}amsey numbers and {Z}arankiewicz numbers.
\newblock {\em Discrete Mathematics}, 219(1):85--95, 2000.

\bibitem{GLL12}
J.~Griggs, W.-T. Li, and L.~Lu.
\newblock Diamond-free families.
\newblock {\em J. Combinatorial Theory (ser. A)}, 119:310--322, 2012.

\bibitem{GL13}
J.~R. Griggs and W.-T. Li.
\newblock The partition method for poset-free families.
\newblock {\em Journal of Combinatorial Optimization}, 25(4):587--596, 2013.

\bibitem{GL15}
J.~R. Griggs and W.-T. Li.
\newblock Poset-free families and lubell-boundedness.
\newblock {\em Journal of Combinatorial Theory, Series A}, 134:166--187, 2015.

\bibitem{GL09}
J.~R. Griggs and L.~Lu.
\newblock On families of subsets with a forbidden subposet.
\newblock {\em Combinatorics, Probability and Computing}, 18(05):731--748,
  2009.

\bibitem{GMT14}
D.~Gr{\'o}sz, A.~Methuku, and C.~Tompkins.
\newblock An improvement of the general bound on the largest family of subsets
  avoiding a subposet.
\newblock {\em arXiv preprint arXiv:1408.5783}, 2014.

\bibitem{GRS99}
D.~S. Gunderson, V.~R{\"o}dl, and A.~Sidorenko.
\newblock Extremal problems for sets forming boolean algebras and complete
  partite hypergraphs.
\newblock {\em Journal of Combinatorial Theory, Series A}, 88(2):342--367,
  1999.

\bibitem{hattingh1998bipartite}
J.~Hattingh and M.~Henning.
\newblock Bipartite {R}amsey theory.
\newblock {\em Utilitas Mathematica}, 53:217--230, 1998.

\bibitem{irving1978bipartite}
R.~W. Irving.
\newblock A bipartite {R}amsey problem and the {Z}arankiewicz numbers.
\newblock {\em Glasgow Mathematical Journal}, 19(01):13--26, 1978.

\bibitem{JLM13}
T.~Johnston, L.~Lu, and K.~G. Milans.
\newblock Boolean algebras and {L}ubell functions.
\newblock {\em Journal of Combinatorial Theory Series A}, July.
\newblock to appear.

\bibitem{kierstead1987ramsey}
H.~A. Kierstead and W.~T. Trotter.
\newblock A ramsey theoretic problem for finite ordered sets.
\newblock {\em Discrete Mathematics}, 63(2):217--223, 1987.

\bibitem{KMY13}
L.~Kramer, R.~Martin, and M.~Young.
\newblock On diamond-free subposets of the {B}oolean lattice.
\newblock {\em J. Combin. Theory Ser. A}, 120(3):545--560, 2013.

\bibitem{mccolm1991ramseyian}
G.~L. McColm.
\newblock A ramseyian theorem on products of trees.
\newblock {\em Journal of Combinatorial Theory, Series A}, 57(1):68--75, 1991.

\bibitem{MSW15}
K.~Milans, D.~Stolee, and D.~West.
\newblock Ordered {R}amsey theory and track representations of graphs.
\newblock {\em Journal of Combinatorics}, 6(4):445--456, 2015.

\bibitem{MS14}
G.~Moshkovitz and A.~Shapira.
\newblock {R}amsey theory, integer partitions and a new proof of the {E}rd{\H
  o}s-{S}zekeres theorem.
\newblock {\em Advances in Mathematics}, 262:1107--1129, Sept. 2014.

\bibitem{nevsetvril1984combinatorial}
J.~Ne{\v{s}}et{\v{r}}il and V.~R{\"o}dl.
\newblock Combinatorial partitions of finite posets and latticesÑramsey
  lattices.
\newblock {\em Algebra Universalis}, 19(1):106--119, 1984.

\bibitem{Sperner28}
E.~Sperner.
\newblock Ein {S}atz {\"u}ber {U}ntermengen einer endlichen {M}enge.
\newblock {\em Mathematische Zeitschrift}, 27:544--548, 1928.

\bibitem{Sage}
W.~Stein et~al.
\newblock Sage: Open source mathematical software.
\newblock {\em 7 December 2009}, 2008.

\bibitem{trotter1999ramsey}
W.~Trotter.
\newblock Ramsey theory and partially ordered sets.
\newblock {\em Contemporary Trends in Discrete Mathmatics, RL Graham, et al.,
  eds., DIMACS Series in Discrete Mathematics and Theoretical Computer
  Science}, 49:337--347, 1999.

\bibitem{West96}
D.~West.
\newblock {\em Introduction to graph theory}.
\newblock Prentice Hall, Inc., Upper Saddle River, NJ, 1996.

\end{thebibliography}

\end{document}